\documentclass[12pt]{amsart}
\usepackage[all]{xy}
\usepackage[parfill]{parskip}
\usepackage{amsmath}
\usepackage{amsfonts}
\usepackage{mathrsfs}
\usepackage[abs]{overpic}

\usepackage{palatino} 

\usepackage{hyperref}

\usepackage{verbatim}
\usepackage{color}

\usepackage{amsmath, amscd, graphicx, latexsym, hyperref, times, calc}
\usepackage{color,graphicx, rotate}
\usepackage[abs]{overpic}
\usepackage{tikz}
\usepackage{tikz-cd}
\usetikzlibrary{arrows, patterns}

\newtheorem{dummy}{anything}[section]
\newtheorem{theorem}[dummy]{Theorem}

\newtheorem{lemma}[dummy]{Lemma}

\newtheorem{proposition}[dummy]{Proposition}

\newtheorem{corollary}[dummy]{Corollary}

\theoremstyle{definition}

\newtheorem{example}[dummy]{Example}

\newtheorem{remark}[dummy]{Remark}

\newtheorem*{acknowledgements}{Acknowledgements}

\DeclareMathOperator{\tor}{tor}
\DeclareMathOperator{\Tight}{Tight}
\DeclareMathOperator{\tw}{tw}
\DeclareMathOperator{\tb}{tb}
\DeclareMathOperator{\rot}{rot}
\DeclareMathOperator{\PD}{PD}

\makeatletter
\newcommand*\bigcdot{\mathpalette\bigcdot@{0.6}}
\newcommand*\bigcdot@[2]{\mathbin{\vcenter{\hbox{\scalebox{#2}{$\m@th#1\bullet$}}}}}
\makeatother

\newcommand{\N}{\mathbb{N}}
\newcommand{\Z}{\mathbb{Z}}

\newcommand{\q}{\mathbf{q}}

\newcommand{\e}{\textrm{e}}
\newcommand{\dfn}[1]{{\em #1}}

\def\bdy{\partial}

\usepackage[utf8]{inputenc}

\title{Non-loose torus knots in $S^1\times S^2$}

\author{Jiaxin Huang}
\address{School of Mathematical Sciences, Shanghai Jiao Tong University, Shanghai, China}
\email{jiaxin\_h@sjtu.edu.cn}

\author{Youlin Li}
\address{School of Mathematical Sciences, Shanghai Jiao Tong University, Shanghai, China}
\email{liyoulin@sjtu.edu.cn}

\author{Zaiting Xu}
\address{School of Mathematical Sciences, Peking University, Beijing, China}
\email{xuzaiting@stu.pku.edu.cn}

\begin{document}

\maketitle

\begin{abstract}
In this paper, we present a complete coarse classification of non-loose Legendrian and transverse torus knots in any contact structure on $S^1\times S^2$.  
\end{abstract}

\section{Introduction}

A Legendrian or transverse knot in an overtwisted contact 3-manifold is called {\it non-loose} if its complement is tight. Two Legendrian knots $L_0$ and $L_1$ in a contact 3-manifold $(M,\xi)$ are {\it coarsely equivalent} if there exists a contactomorphism of $(M,\xi)$ which is smoothly isotopic to the identity and maps $L_0$ to $L_1$. 

The study of non-loose Legendrian and transverse knots and links in contact 3-manifolds has become an increasingly active area of research in contact topology. Examples include the trivial knot in $S^3$ \cite{ef, v}, torus knots in $S^3$ \cite{go, ma, emm} and in lens spaces \cite{emx},  rational unknots in lens spaces \cite{go2, cemm}, the Hopf link and the connected sum of two Hopf links in $S^3$ \cite{go1, lo}, as well as  Hopf links in lens spaces  \cite{cgo, c}. Such investigations are deeply intertwined with the study of tight contact structures on the link complements, and provide valuable tools for constructing and classifying tight contact structures on certain closed 3-manifolds \cite{emtv}.

In this paper, we investigate the non-loose Legendrian and transverse torus knots in contact $S^1\times S^2$.  In contrast to previous results concerning knots in 3-sphere and lens spaces,  the (rational) rotation number of a Legendrian knot in a contact $S^1\times S^2$ is generally not well-defined. 

Let $S^1\times S^2:=\{(\theta, x_1, x_2, x_3)\in S^1\times \mathbb{R}^3\mid \theta\in S^1, x_1^2+x_2^2+x_3^2=1\}$. Define $T_0 = \{(\theta, x_1, x_2, x_3) \in S^1 \times S^2 \mid x_3 = 0\}$. Then $T_0$ is a Heegaard torus, meaning that the closures of the components of $(S^1 \times S^2)\setminus T_0$ are two solid tori. Let $V_1$ be the solid torus $\{(\theta, x_1, x_2, x_3) \in S^1 \times S^2 \mid x_3 \leq 0\}$ and $V_2$ be the solid torus $\{(\theta, x_1, x_2, x_3) \in S^1 \times S^2 \mid x_3 \geq 0\}$. A knot in $S^1 \times S^2$ is called a torus knot if it is smoothly isotopic to a knot on $T_0$.  Consider the curve on $T_0 = \partial V_1$ given by $\theta = \theta_0$, oriented positively in the $x_1 x_2$-plane. This is a meridian of $V_1$, and we denote its homology class in $H_1(T_0)$ by $m_0$. Similarly, the curve given by $(x_1, x_2, x_3) = (1, 0, 0)$, oriented
by the parameter $\theta$, is a longitude, and we denote its homology class in $H_1(T_0)$ by $l_0$. Then $(l_0, m_0)$ is a basis for $H_1(T_0)$. An oriented knot in $S^1 \times S^2$ is called a $(p, q)$-torus knot if it is isotopic to an oriented knot on $T_0$ homologically equivalent
to $pm_0 + ql_0$, where $p$ and $q$ coprime. A $(\pm1, 0)$-torus knot is trivial, i.e. it bounds a
disk in $S^1 \times S^2$. A $(p, 1)$-torus knot is isotopic to $S^1 \times \{(0, 0, 1)\}$ oriented by the
parameter $\theta$.

By \cite[Theorem 4.10.1]{g}, there is a unique tight contact structure on $S^1\times S^2$, which we denote by $(S^1\times S^2, \xi_{std})$. In \cite{cdl}, Chen, Ding, and the second author classified the Legendrian torus knots in $(S^1\times S^2, \xi_{std})$ up to Legendrian isotopy. 

By Eliashberg's classification \cite{e}, the isotopy classes of overtwisted contact structures on $S^1\times S^2$ are in one-to-one correspondence with the homotopy classes of 2-plane fields. Since $H_{1}(S^1\times S^2)\cong\mathbb{Z}$ contains no 2-torsion, two 2-plane fields correspond to the same $spin^c$ structure if and only if their Euler classes (which are even elements in $\mathbb{Z}$) coincide. As explained in  \cite[Section 7.5]{eko}, all 2-plane fields in a fixed $spin^c$ structure can be obtained via connected summing the overtwisted contact structures on $S^3$. That is, the overtwisted contact structures on $S^1\times S^2$ with given Euler class are in one-to-one correspondence with the integer set $\mathbb{Z}$.

By \cite[Corollary 1.6]{cemm}, a Legendrian $(p, \pm1)$-torus knot in a contact $S^1\times S^2$ cannot be non-loose. We therefore restrict to the case $|q|\geq 2$. According to \cite[Proposition 2.2]{cdl}, for $|q|\geq 2$, a $(p_1, q)$-torus knot and a $(p_2, q)$-torus knot are smoothly isotopic in $S^1 \times S^2$ if and only if $p_2\equiv \pm p_1 \pmod {2q}$. Without loss of generality, we assume that $-q>p>0$. 

Suppose
\[
\frac{q}{p} = [a_1, \dots, a_m] = a_1 - \cfrac{1}{a_2 - \cfrac{1}{\cdots - \cfrac{1}{a_{m-1} - \cfrac{1}{a_m}}}},
\]
where $a_i\leq -2$ for $1\leq i\leq m$.  

Further, suppose $$(-1-\frac{p}{q})^{-1}=[b_1, \ldots, b_n],$$
where $b_i\leq -2$ for $1\leq i\leq n$.

Let $$m(p,q)=|(a_{1}+1)\cdots(a_{m-1}+1)a_{m}|\cdot|(b_{1}+1)\cdots(b_{n-1}+1)b_{n}|,$$
$$n(p,q)=|(a_{1}+1)\cdots(a_{m-1}+1)(a_{m}+1)|\cdot|(b_{1}+1)\cdots(b_{n-1}+1)(b_{n}+1)|,$$
and $$l(p,q)=|(a_{1}+1)\cdots(a_{m-2}+1)(a_{m-1}+1)|\cdot|(b_{1}+1)\cdots(b_{n-2}+1)(b_{n-1}+1)|,$$
with the convention that $|(a_{1}+1)\cdots(a_{m-1}+1)|=1$ when $m=1$ and $|(b_{1}+1)\cdots(b_{n-1}+1)|=1$ when $n=1$.


According to \cite[Proposition 2.1]{cdl}, for a Legendrian $(p,q)$-torus knot $L$ in a contact $S^1 \times S^2$, the framing induced by any Heegaard torus containing $L$ is independent of the choice of such a torus. This allows us to define the \textit{twisting number} $\tw(L)$ of $L$ as the difference between its contact framing and the framing induced by a Heegaard torus. We denote by $S_{+}(L)$ and $S_{-}(L)$ the positive and negative stabilizations of $L$, respectively. It follows immediately that $\tw(S_{\pm}(L))=\tw(L)-1$. 
We also employ this notation without signs to denote the stabilization of a transverse knot. 

Let $(M,\xi)$ be a contact 3-manifold and $[T]$ an isotopy class of embedded tori in $M$. The \textit{convex Giroux torsion} of $(M,\xi)$ along $[T]$ is the supremum of $t\in \frac{1}{2}\mathbb{N}\cup\{0\}$ for which there is a contact embedding of $$(T^2\times [0,1], \ker(\sin(2t\pi z)dx+\cos(2t\pi z)dy))$$ into $(M,\xi)$, with $T^{2}\times \{z\}$ being in the class $[T]$. For a Legendrian or transverse torus knot $L$ in a contact $S^1\times S^2$, we say $\tor(L)=t$ if the complement of the standard neighborhood of $L$ has convex Giroux torsion $t$ along the isotopy class of the boundary torus.



Before presenting the classification results, we establish the complete range of Euler classes of contact structures that support non-loose torus knots. 

\begin{theorem}\label{thm:Eulerrange} 
    Suppose $-q>p>0$. Let $\xi$ and $\xi^T$ be any overtwisted contact structures on $S^1\times S^2$ supporting non-loose Legendrian and transverse $(p,q)$-torus knots with $\tor=t$, respectively. Then $e(\xi)=\pm2k$ and $\e(\xi^T)=2k$ with 
    \[k\in
    \begin{cases}
       \{1,2,\ldots,-q-1\}, &\text{ if } t=0,\\
       \{k_e+q+1,k_e+q+2,\ldots,-q-1\}, &\text{ if } t\in\N,\\
       \{k_e+2q+1,k_e+2q+2,\ldots,-1\}, &\text{ if } t\in\frac{1}{2}\N\setminus\N,\\
    \end{cases}
    \]
where \[k_e=
    \begin{cases}
    -q+(a_{m}+2)(q-q'), &\text{if }a_m<-2,\\
    -q+(b_{n}+2)q', &\text{if }a_m=-2.
    \end{cases}\]
    In particular, if a contact structure $\xi$ on $S^1\times S^2$ contains a non-loose $(p,q)$-torus knot, then $|e(\xi)|\leq -2q-2$.
\end{theorem}

\begin{remark}
    For any contact structure $\xi$ on $S^1\times S^2$ supporting Legendrian torus knots with tight complement, $e(\xi)=0$ if and only if $\xi$ is tight.  Theorem~\ref{thm:Eulerrange} is a generalization of \cite[Corollary 1.6]{cemm}.
\end{remark}

We classify the non-loose Legendrian $(p,q)$-torus knots with $\tor=0$ in contact $S^1\times S^2$ as follows.

\begin{theorem}\label{thm:tw>0}
Suppose $-q>p>0$ and $i>0$ are integers. The number of non-loose Legendrian $(p,q)$-torus knots with $\tw=i$ and $\tor=0$ in some contact $S^1\times S^2$ is exactly $2n(p,q)$, and any such Legendrian knot can be realized as a Legendrian knot shown in Figure~\ref{Figure:tw>0}.  These knots are denoted by $$L^{i}_{\pm, k}~\text{ for }~1\leq k\leq n(p,q).$$ Moreover, the following stabilization properties hold: $$S_{\pm}(L^{i}_{\pm,k})=L^{i-1}_{\pm,k}~\text{ for }~ i>1, S_{\mp}(L^{i}_{\pm,k})~\text{is loose for}~ i>0, ~\text{and}$$ $$S_{\pm}^{j}(L^{1}_{\pm,k})~\text{is non-loose for any}~ j>0.$$
\end{theorem}

\begin{figure}[htb]
\begin{overpic}
{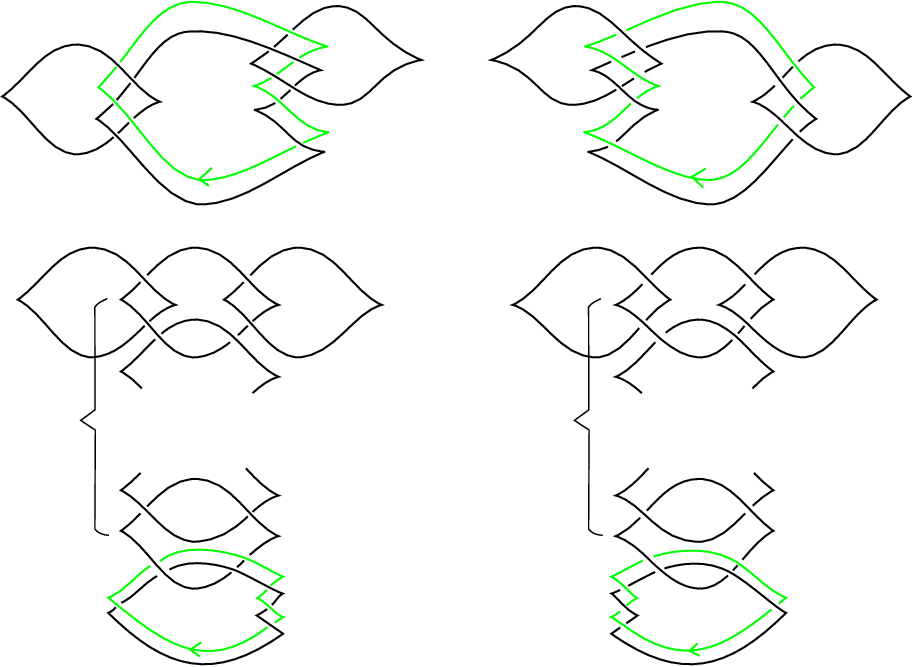}
\put(180, 265){$(\frac{q'}{q'-q})$}
\put(0, 305){$(\frac{q'-q}{q'})$}
\put(140, 230){$(+1)$}
\put(50, 310){$L_{+}$}

\put(230, 265){$(\frac{q'-q}{q'})$}
\put(410, 305){$(\frac{q'}{q'-q})$}
\put(275, 230){$(+1)$}
\put(370, 310){$L_{-}$}

\put(180, 160){$(\frac{q'}{q'-q})$}
\put(0, 200){$(\frac{q'-q}{q'})$}
\put(95, 100){$\cdot$}
\put(95, 120){$\cdot$}
\put(95, 110){$\cdot$}
\put(140, 10){$(+1)$}
\put(140, 60){$(-1)$}
\put(140, 80){$(-1)$}
\put(140, 130){$(-1)$}
\put(80, 186){$(-1)$}
\put(40, 40){$L_{+}$}
\put(10, 115){$i-1$}
\put(0, -5){$\text{Euler class}<0$.}

\put(235, 195){$(\frac{q'-q}{q'})$}
\put(415, 160){$(\frac{q'}{q'-q})$}
\put(335, 100){$\cdot$}
\put(335, 120){$\cdot$}
\put(335, 110){$\cdot$}
\put(380, 10){$(+1)$}
\put(375, 60){$(-1)$}
\put(375, 80){$(-1)$}
\put(375, 130){$(-1)$}
\put(320, 186){$(-1)$}
\put(280, 40){$L_{-}$}
\put(248, 115){$i-1$}
\put(230, -5){$\text{Euler class}>0$.}

\end{overpic}
\vspace{1mm}
\caption{Half of the $2n(p,q)$ non-loose Legendrian torus knots $T_{p,q}$ with $\tw=1$ and $\tor=0$ are displayed in the top left diagram, while the remaining half are shown in the top right diagram. Similarly, for cases where $\tw=i>1$ and $\tor=0$, half of the $2n(p,q)$ non-loose Legendrian torus knots $T_{p,q}$ appear in the bottom left diagram, with the other half presented in the bottom right diagram. Here $\frac{q'}{p'}$ ($p'>0$) denotes the largest rational number satisfying $pq'-p'q=1$. The Euler class is negative for the contact structures on the left and positive for those on the right. }
\label{Figure:tw>0}
\end{figure}

\begin{theorem}\label{thm:tw=0}
Suppose $-q>p>0$. The number of non-loose Legendrian $(p,q)$-torus knots with $\tw=0$ and $\tor=0$ in some contact $S^1\times S^2$ is exactly $m(p,q)-2$. Moreover, any such Legendrian knot can be realized as one of the knots shown in Figure~\ref{Figure:tw=0}.  
\end{theorem}

\begin{remark} \label{rmk:tw=0intight}
As a corollary of the main theorem in  \cite{cdl}, we know that there are exactly $2$  Legendrian $(p,q)$-torus knots with $\tw=0$ in $(S^1\times S^2, \xi_{std})$ up to coarse equivalence. 
\end{remark}

\begin{figure}[htb]
\begin{overpic}
[scale=0.75]
{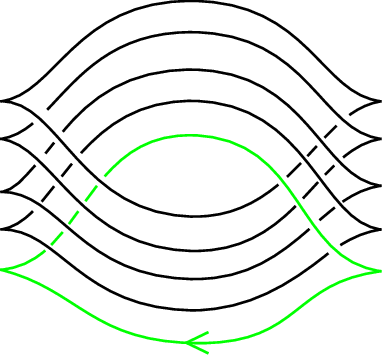}
\put(142, 92){$(-\frac{q}{q'})$}
\put(142, 75){$(-\frac{q}{q-q'})$}
\put(142, 57){$(+1)$}
\put(142, 42){$(+1)$}
\put(150, 20){$L$}
\end{overpic}
\caption{Non-loose Legendrian torus knots in $S^1\times S^2$ with $\tw=0$ and $\tor=0$. }
\label{Figure:tw=0}
\end{figure}

The classification results can be qualitatively described using mountain ranges, where each integer lattice point corresponds to a unique non-loose Legendrian knot, and its vertical coordinate gives the twisting number. While the rotation number is not well-defined here, we can still construct similar diagrams by modeling stabilizations as downward moves: positive stabilization corresponds to moving down a line of slope $-1$, and negative stabilization to moving down a line of slope $1$.

\begin{theorem}\label{thm:Legmountain}
The mountain range of non-loose Legendrian $(p,q)$-torus knots with $\tor=0$ consists of $n(p,q)$ pairs of components, each having the form shown in Figure~\ref{Figure:mountain}. In each such pair, the left component contains a line of slope $-1$ on its far left, and the right component contains a line of slope $+1$ on its far right; both of these lines can be extended infinitely upward. Moreover, the peaks occur at the horizontal line $\tw=0$.
\end{theorem}


\begin{figure}[htb]
\begin{overpic} 
[scale=0.7]
{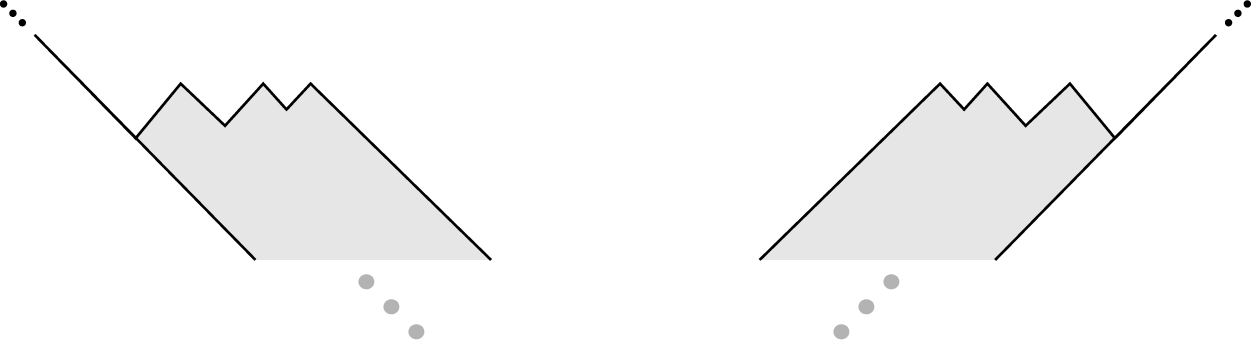}
\put(30, 100){$\text{Euler class}=-e<0.$}
\put(280, 100){$\text{Euler class}=e>0.$}
\end{overpic}
\caption{A pair of components in the mountain ranges for non-loose Legendrian $(p,q)$–torus knots with $\tor = 0$. The two components are symmetric and correspond to nonzero Euler classes with opposite signs. The left component lies in a contact $S^1 \times S^2$ with Euler class $-e < 0$, while the right component lies in a contact $S^1 \times S^2$ with Euler class $e > 0$.  
Each point is realized by a unique non-loose Legendrian knot. }
\label{Figure:mountain}
\end{figure}

The precise classification of non-loose Legendrian $(p,q)$-torus knots with $\tw\leq 0$ and $\tor=0$ will be carried out in Proposition~\ref{prop:wings} and Corollary~\ref{coro:numberofwings}.


We classify the non-loose Legendrian $(p,q)$-torus knots with $\tor>0$ in contact $S^1\times S^2$ as follows.
\begin{theorem}\label{thm:tor>0}
Suppose $-q>p>0$ and $0<t\in \frac{1}{2}\mathbb{N}$. The number of non-loose Legendrian $(p,q)$-torus knots in some contact $S^1\times S^2$ with $\tw=i$ and $\tor=t$ is exactly $2l(p,q)$.  These knots are denoted by $$L^{i, t}_{\pm, k}~\text{ for }~1\leq k\leq l(p,q).$$ Moreover, $$S_{\pm}(L^{i,t}_{\pm,k})=L^{i-1, t}_{\pm,k}~\text{ for }~ i\in\mathbb{Z},~\text{and}~S_{\mp}(L^{i, t}_{\pm,k})~\text{is loose for}~ i\in \mathbb{Z}.$$ 
\end{theorem}

\begin{corollary}\label{coro:destabilization}
    Any non-loose Legendrian $(p,q)$-torus knot destabilizes except when $\tw=0$. Non-loose Legendrian knots with $\tw=0$ sometimes destabilize and sometimes do not.
\end{corollary}

Building on the classification results of Legendrian knots, we now derive the corresponding results for transverse knots. The non-loose transverse torus knots are precisely the transverse push-offs of non-loose Legendrian torus knots that remain non-loose after any number of negative stabilizations. By considering transverse push-offs of the Legendrian knots shown in the right subfigure of Figure~\ref{Figure:mountain} -or equivalently, modulo negative stabilizations- we can represent the  range of non-loose transverse $(p,q)$-torus knots with $\tor=0$ as the union of the components illustrated in Figure~\ref{Figure:mountain1}.  

\begin{corollary}\label{coro:transmountain}
The non-loose transverse $(p,q)$-torus knots with $\tor=0$ can be described as the disjoint union of $n(p,q)$ intervals, each depicted in Figure~\ref{Figure:mountain1}.  
\end{corollary}

\begin{figure}[htb]
\begin{overpic} 
[scale=0.9]
{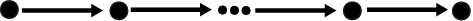}
\put(220, 2){$\text{Euler class}>0.$}
\end{overpic}
\caption{An interval for non-loose transverse $(p,q)$–torus knots with $\tor = 0$. It lies in a contact $S^1 \times S^2$ with Euler class $e > 0$.  
Each point is realized by a unique non-loose transverse knot. Each arrow stands for a transverse stabilization.}
\label{Figure:mountain1}
\end{figure}

The precise classification of the non-loose transverse $(p,q)$-torus knots in $S^1\times S^2$ with $\tor=0$ will be carried out in Corollary~\ref{coro:transversetor=0}.

\begin{corollary} \label{coro:transversetor>0}
    Suppose $-q>p>0$ and $0<t\in \frac{1}{2}\mathbb{N}$. The number of non-loose transverse $(p,q)$-torus knot in $S^1\times S^2$ with $\tor=t$ is exactly $l(p,q)$.  These knots are denoted by $$T^{t}_{j}~\text{ for }~1\leq j\leq l(p,q).$$ Moreover, when stabilized $T^t_j$ becomes loose, and $T^{t+\frac{1}{2}}_{j}$ is obtained by a half Lutz twist of $T^{t}_{j}$.
\end{corollary}



As a special case,  we give an explicit classification of non-loose Legendrian and transverse $(1,q)$-torus knots in contact $S^1\times S^2$.
\begin{example} \label{example:(1,q)}
    Suppose $q<-1$. The $(1,q)$-torus knot has non-loose Legendrian representatives only in the contact structures on $S^1\times S^2$ with Euler classes $e=\pm2k$ for $k\in\{1,2,\ldots,-q-1\}$. The classification in each of these contact structures is as follows.
     \begin{enumerate} 
     \item In contact $S^1\times S^2$ with $e=\mp2k$, the non-loose Legendrian $(1,q)$-torus knots are $L_{\pm,k}^{i}$ for $i\in\Z$ and $k\in\{1,2,\ldots,-q-2\}$, with $$\tw(L_{\pm,k}^{i})=i\text{ and }\tor(L_{\pm,k}^{i})=0 ,$$
     such that $$S_{\pm}(L^{i}_{\pm,k})=L^{i-1}_{\pm,k}\text{ and }S_{\mp}(L^{i}_{\pm,k})~\text{is loose} .$$
     Note that this case will not occur when $q=-2$.
     \item In contact $S^1\times S^2$ with $e=\pm 2(q+1)$, the non-loose Legendrian torus knots are $L_{\pm, -q-1}^{i,t}$ for $i\in \Z$ and $t\in \N\cup\{0\}$, with $$\tw(L_{\pm, -q-1}^{i,t})=i\text{ and }\tor(L_{\pm, -q-1}^{i,t})=t,$$ such that $$S_{\pm}(L^{i,t}_{\pm,-q-1})=L^{i-1, t}_{\pm,-q-1}\text{ and }S_{\mp}(L^{i, t}_{\pm,-q-1})~\text{is loose}.$$ 
     \item In contact $S^1\times S^2$ with $e=\pm 2$, the non-loose Legendrian torus knots are $L_{\pm, -q-1}^{i,t+\frac{1}{2}}$ for $i\in \Z$ and $t\in\N\cup\{0\}$, with $$\tw(L_{\pm, -q-1}^{i,t+\frac{1}{2}})=i\text{ and }\tor(L_{\pm, -q-1}^{i,t+\frac{1}{2}})=t+\frac{1}{2},$$ such that $$S_{\pm}(L^{i,t+\frac{1}{2}}_{\pm,-q-1})=L^{i-1, t+\frac{1}{2}}_{\pm,-q-1}\text{ and }S_{\mp}(L^{i, t+\frac{1}{2}}_{\pm,-q-1})~\text{is loose}.$$ 
     \end{enumerate}
\end{example}

See Figure~\ref{Figure:p=1non-loose} for the mountain ranges of non-loose Legendrian $(1,q)$-torus knots in $S^1\times S^2$.
\begin{figure}[htb]
\begin{overpic} 
[scale=0.7]
{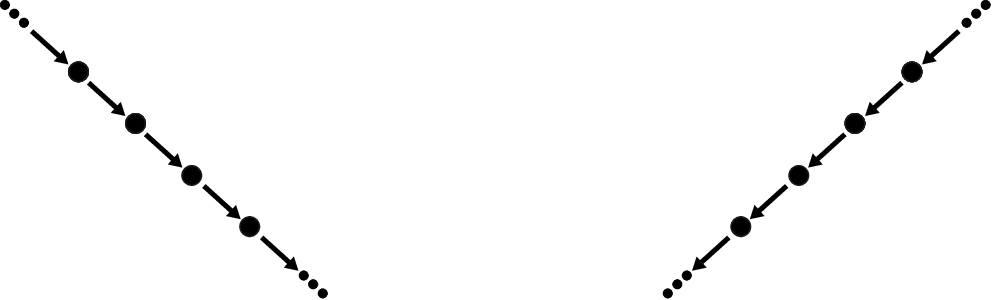}
\put(30, 90){$\text{Euler class}=-e<0.$}
\put(200, 90){$\text{Euler class}=e>0.$}
\end{overpic}
\caption{A pair of components in the mountain ranges of non-loose Legendrian $(1,q)$-torus knot with $\tor=0$. The two components are symmetric and correspond to nonzero Euler classes with opposite signs. The left component lies in a contact $S^1 \times S^2$ with Euler class $-e < 0$, while the right component lies in a contact $S^1 \times S^2$ with Euler class $e > 0$.} 
\label{Figure:p=1non-loose}
\end{figure}

\begin{example} \label{example:(1,q)transverse}
    Suppose $q<-1$. The $(1,q)$-torus knot has non-loose transverse representatives only in $S^1\times S^2$ with Euler class $e=2k$ for $k\in\{-1\}\cup\{1,2,\ldots,-q-1\}$. The classification in each of these contact structures is as follows.
    \begin{enumerate}
     \item In contact $S^1\times S^2$ with $e=2k$, the non-loose transverse $(1,q)$-torus knots are $T_{k}$ for $k\in\{1,2,\ldots,-q-2\}$, with $\tor(T_{k})=0$, and when stabilized $T_k$ becomes loose.
     Note that this case will not occur when $q=-2$.
     \item In contact $S^1\times S^2$ with $e=-2q-2$, the non-loose transverse $(1,q)$-torus knots are $T_{ -q-1}^{t}$ for $t\in\N\cup\{0\}$, with $\tor(T_{-q-1}^{t})=t$, and when stabilized $T_{-q-1}^t$ becomes loose. Moreover, $T^{t+1}_{-q-1}$ is obtained by a full Lutz twist of $T^{t}_{-q-1}$ for $t\geq 0$.
     \item In contact $S^1\times S^2$ with $e=-2$, the non-loose transverse $(1,q)$-torus knots are $T_{-q-1}^{t+\frac{1}{2}}$ for $t\in\N\cup\{0\}$, with $\tor(T_{-q-1}^{t+\frac{1}{2}})=t+\frac{1}{2}$, and when stabilized $T_{-q-1}^{t+\frac{1}{2}}$ becomes loose. Moreover, $T^{t+\frac{1}{2}}_{-q-1}$ is obtained by a half Lutz twist of $T^{t}_{-q-1}$ for $t\geq 0$.
     \end{enumerate}
\end{example}

\begin{acknowledgements}
    The authors would like to thank Hyunki Min for valuable comments. The first and second authors were partially supported by Grant No. 12271349 from the National Natural Science Foundation of China. The third author was partially supported by National Key R\&D Program of China (No.2020YFA0712800) and Grant No. 12131009 of the NSFC.
\end{acknowledgements}

\section{Background and preliminaries}\label{section:pre}


\subsection{Farey graph} In this paper, we assume that every point \(q/p\) in the Farey graph satisfies \(p \geq 0\). For a point \(q/p\) in the Farey graph, define:
\begin{itemize}
    \item \(\left(q/p\right)^c = q'/p'\) as the farthest clockwise point from \(q/p\) that is greater than \(q/p\) and connected to \(q/p\) by an edge;
    \item \(\left(q/p\right)^a = q''/p''\) as the farthest anti-clockwise point from \(q/p\) that is less than \(q/p\) and connected to \(q/p\) by an edge.
\end{itemize}

For two points \(b/a\) and \(d/c\) in the Farey graph, define their \textit{dot product} as
\[
\frac{b}{a} \bigcdot \frac{d}{c} = bc - ad,
\] and if $bd\geq0$, then the \textit{Farey sum} is defined as
\[
\frac{b}{a} \oplus \frac{d}{c} = \frac{b + d}{a + c}.
\]

Note that \(\frac{q'}{p'} \oplus \frac{q''}{p''} = \frac{q}{p}\).

\subsection{Tight contact structures on thickened tori and solid tori.}\label{section:tori}
As in \cite[Section 2]{emm}, we denote by $\Tight^{min}(T^2\times[0,1]; s_0, s_1)$ the set of isotopy classes of minimally twisting tight contact structures on $T^2\times[0,1]$ with convex boundary, where $T^2\times \{i\}$ has two dividing curves of slope $s_i$ for $i=0,1$. We further denote by $\Tight(S_s; r)$ (respectively $\Tight(S^s; r)$) the set of isotopy classes of tight contact structures on the solid torus $S_s$ (respectively $S^s$) with lower (respectively upper) meridian of slope $s$ and convex boundary having two dividing curves of slope $r$.

\begin{lemma} \cite{h}\label{lemma:upper0}
Suppose $-q>p>0$ and $\frac{q}{p}=[a_1,\ldots, a_m]$, where $a_i\leq -2$ for $1\leq i\leq m$. Then we have $$|\Tight(S^0; \frac{q}{p})| =|(a_{1}+1)\cdots(a_{m-1}+1)a_{m}|.$$    
\end{lemma}

\begin{lemma}\label{lemma:lower0}
Suppose $-q>p>0$ and $(-1-\frac{p}{q})^{-1}=[b_1, \ldots, b_n],$ where $b_i\leq -2$ for $1\leq i\leq n$. Then we have $$|\Tight(S_0; \frac{q}{p})| =|\Tight(S^{\infty}; \frac{p}{q})| =|(b_{1}+1)\cdots(b_{n-1}+1)b_{n}|.$$ 
\end{lemma}
\begin{proof}
The first equality follows from the existence of an orientation-preserving self-diffeo-morphism of $T^2\times[0,1]$ that acts by $\begin{pmatrix} 
0 & 1  \\
1 & 0
\end{pmatrix}$ on $T^2$ and reverses the interval $[0,1]$. 

By cutting the solid torus along a meridian disk, adding one full twist, and regluing, we obtain a diffeomorphism showing that $|\Tight(S^{\infty}; \frac{p}{q})|=|\Tight(S^{\infty}; \frac{p}{q} +1)|$. Furthermore, via an orientation-preserving self-diffeomorphism of $T^2\times[0,1]$ that acts by $\begin{pmatrix} 
0 & 1  \\
-1 & 0
\end{pmatrix}$ on $T^2$ and is the identity on $[0,1]$, we have $|\Tight(S^{\infty}; \frac{p}{q} +1)|=|\Tight(S^{0}; (-\frac{p}{q} -1)^{-1})|.$ The second equality then follows from Lemma~\ref{lemma:upper0}.
\end{proof}

\begin{lemma} \label{lemma:L(q,p)}
    With the notation above, we have $$|\Tight(L(-q,p))|=|(a_{1}+1)\cdots(a_{m-1}+1)(a_{m}+1)|$$ and 
    $$|\Tight(L(q,p))|=|(b_{1}+1)\cdots(b_{n-1}+1)(b_{n}+1)|.$$
In particular, $n(p,q)$ is the number of tight contact structures on $L(q,p)\#L(-q,p)$. 
\end{lemma}
\begin{proof}
    Note that $$|\Tight(L(-q,p))|=|\Tight(L_{q/p}^0)|=|\Tight(S^0;(\frac{q}{p})^c)|$$ and $$|\Tight(L(q,p))|=|\Tight(L^{q/p}_0)|=|\Tight(S_0;(\frac{q}{p})^a)|.$$

    We obtain the first equality by Lemma~\ref{lemma:upper0}, since $(q/p)^c=[a_1,\ldots,a_{m-1},a_m+1]$.
    
    In the proof of Lemma~\ref{lemma:lower0}, we actually use the orientation-preserving self-diffeomorphism of $T^2\times[0,1]$ that is $\begin{pmatrix} 
    1 & 0  \\
    -1 & -1
    \end{pmatrix}$ on $T^2$ and reverses the interval $[0,1]$, to transform the minimal path from $0$ anti-clockwise to $q/p$ into a minimal path from $(-1-\frac{p}{q})^{-1}$ clockwise to $0$. Notice that $(q/p)^a$ is the second-to-last vertex in the former path, while $((-1-\frac{p}{q})^{-1})^c$ is the second vertex in the latter path. Hence, the same diffeomorphism will transform the minimal path from $0$ anti-clockwise to $(q/p)^a$ into a minimal path from $((-1-\frac{p}{q})^{-1})^c$ clockwise to $0$. Then the second equality follows from Lemma~\ref{lemma:lower0}.
\end{proof}

\subsection{Pairs of paths describing contact structures on \texorpdfstring{$S^1\times S^2$}{S1 x S2}}\label{section:paths}
We can describe contact structures on $S^1\times S^2$ using pairs of paths in the Farey graph. Given any rational number $q/p$ with $-q > p > 0$, let $P_1$ be a path that describes a contact structure on a solid torus $V_1$ with lower meridian $0$ and boundary slope $q/p$, i.e. an element in $\Tight(S_{0}; q/p)$. Similarly, let $P_2$ be a path that describes a contact structure on a solid torus $V_2$ with upper meridian $0$ and boundary slope $q/p$, i.e. an element in $\Tight(S^{0}; q/p)$. When we view $S^1\times S^2=L^0_0$ as the union of $V_1$ and $V_2$, the pair of paths $(P_1, P_2)$ describes a contact structure $\xi_{P_1,P_2}$ on $S^1\times S^2$. Let $L_{P_1,P_2}$ be a Legendrian divide on $\bdy V_1=\bdy V_2$. 

\begin{lemma}\label{lem:divide}
Each such Legendrian divide $L_{P_1,P_2}$ can be represented via a contact surgery diagram as shown in Figure~\ref{Figure:tw=0}.
\end{lemma} 

\begin{proof}

Consider the smooth surgery diagram for $S^1 \times S^2$ on the left of Figure~\ref{Figure:ContactS1xS2}. This diagram can be interpreted as being obtained by Dehn filling on $T^2 \times [0,1]$, where the filling is performed along a curve of slope $\frac{q'}{q-q'}$ on $-T^2 \times {0}$ and a curve of slope $\frac{q'}{q-q'}$ on $T^2 \times {1}$. We convert this into the contact surgery diagram shown on the right of Figure~\ref{Figure:ContactS1xS2}. The complement of the Legendrian Hopf link in this diagram is an $I$-invariant neighborhood of a convex torus with two dividing curves of slope $-1$.

Let $\phi_1$ be the diffeomorphism of $T^2$ represented by  the matrix $$\begin{pmatrix}
        q'-q & q' \\ p'-p & p'
    \end{pmatrix}.$$ We have   
    $$\begin{pmatrix}
        q'-q & q' \\ p'-p & p'
    \end{pmatrix}\begin{pmatrix}
        q' \\ q-q'
    \end{pmatrix}=\begin{pmatrix}
        0 \\ -1
    \end{pmatrix},$$and
    $$\begin{pmatrix}
        q'-q & q' \\ p'-p & p'
    \end{pmatrix}\begin{pmatrix}
        -1 \\ 1
    \end{pmatrix}=\begin{pmatrix}
        q \\ p
    \end{pmatrix}.$$ 

Applying the coordinate change via $\phi_1$, we obtain an $I$-invariant neighborhood of a convex torus whose two dividing curves now have slope $q/p$. Consequently, the two solid tori attached along $T^2 \times {0}$ and $T^2 \times {1}$ correspond to an element $P_1$ in $\Tight(S_{0}; q/p)$ and an element $P_2$ in $\Tight(S^{0}; q/p)$, respectively. Finally, applying the Ding-Geiges algorithm from \cite{dg1}, we convert the contact surgery diagram in Figure~\ref{Figure:ContactS1xS2} into the one presented in Figure~\ref{Figure:tw=0}.
\end{proof}

\begin{figure}[htb]
\begin{overpic}
{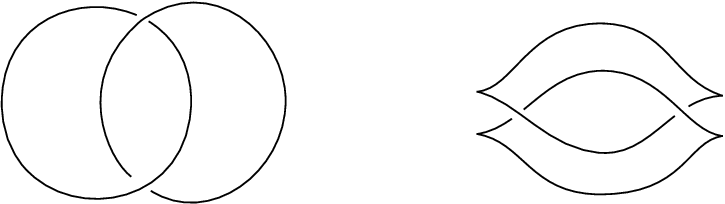}
\put(120, 90){$\frac{q-q'}{q'}$}
\put(-5, 90){$\frac{q'}{q-q'}$}

\put(220, 80){$(\frac{q}{q-q'})$}
\put(220, 10){$(\frac{q}{q'})$}
\end{overpic}
\caption{Left: smooth surgery diagram of $S^1\times S^2$. Right: contact surgery diagram of a contact $S^1\times S^2$. }
\label{Figure:ContactS1xS2}
\end{figure}
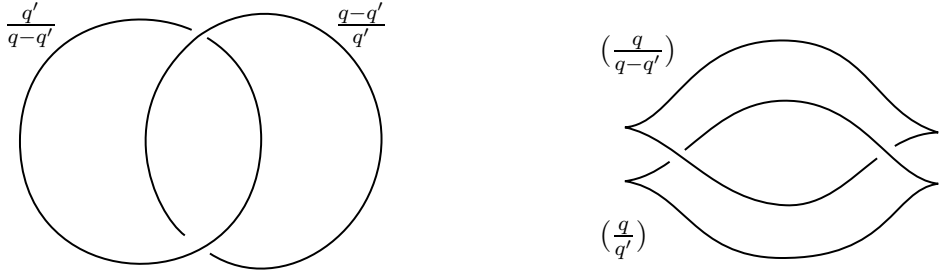

In the Farey graph, $P_1$ is a decorated path from $q/p$ anti-clockwise to $0$ with the last edge from $\infty$ to $0$ unsigned, and $P_2$ is a decorated path from $q/p$ clockwise to $0$ with the last edge from $-1$ to $0$ unsigned. Hence, we will also denote by $P_1$ the decorated path from $q/p$ anti-clockwise to $\infty$, and $P_2$ the decorated path from $q/p$ clockwise to $-1$. Suppose the vertices of \(P_1\) are \(p_1, p_2, \dots, p_k\), where \(p_1 = q/p\) and \(p_k = \infty = -1/0\), and those of \(P_2\) are \(q_1, q_2, \dots, q_l\), where \(q_1 = q/p\) and \(q_l = -1\). By \cite[Lemma 2.9]{emm}, if \(p_i = [b_1, \dots, b_{j-1}, b_j]\), then \(p_{i+1} = [b_1, \dots, b_{j-1}]\) for \(i = 1, \dots, k-2\); and if \(q_i = [c_1, \dots, c_j]\), then \(q_{i+1} = [c_1, \dots, c_j + 1]\).

\begin{lemma}\label{lemma:paths}
In the paths $P_1$ and $P_2$, we have
\begin{enumerate}
\item $p_2=(q/p)^a=q''/p''$ and $q_2=(q/p)^c=q'/p'$, and
\item the numerators of \(p_1, p_2, \dots, p_k\) are strictly increasing, so are those of \(q_1, q_2, \dots, q_l\).
\end{enumerate}    
\end{lemma}

\begin{proof}
 The first part follows from \cite[Lemma 2.1]{emm}. We now  prove the second part.  Suppose $i=1,\ldots, k-1$, then $p_i$ and $p_{i+1}$ are two vertices in a continued fraction block. By the discussion in \cite[Subsection 2.3]{emm}, we have $p_i=p_{i+1}\oplus q_{j}$ for some $j\in\{1,\ldots, l\}$. Since the numerators of $ p_1, p_2, \dots, p_k$ and \(q_1, q_2, \dots, q_l\)  are all negative, the numerator of $p_i$ is strictly less than that of $p_{i+1}$.  On the other hand, suppose $j=1,\ldots, l-1$, we have $q_j=q_{j+1}\oplus p_{i}$ for some $i\in\{1,\ldots, k\}$. For the same reason, the numerator of $q_j$ is strictly less than that of $q_{j+1}$. 
\end{proof}

In the Farey graph, starting from \(q/p\), we divide each path \(P_i\) (\(i = 1, 2\)) into continued fraction blocks. Say \((P_1, P_2) = (P'_1, P'_2)\) if each \(P'_i\) can be obtained from \(P_i\) by shuffling the signs of basic slices within its continued fraction blocks for $i=1,2$. The \dfn{length} of a continued fraction block $C$ is the number of basic slices in $C$, denoted by $|C|$.

\begin{lemma}\label{lem:length1}
    Let $A$ and $B$ be the first continued fraction blocks in $P_1$ and $P_2$ starting at $q/p$, respectively. Then the length of either $A$ or $B$ is $1$. Moreover, $A$ and $B$ both have length $1$ if and only if $q/p=-2$.
\end{lemma}
\begin{proof}
The first part of this lemma follows from \cite[Lemma 2.10]{emm}. We now turn to the second part. If $q/p=-2$, then $P_1=\{-2,\infty\}$ is a continued fraction block, as does  $P_2=\{-2,-1\}$. In this case, we have $|A|=|B|=1$. If $q/p=-n$ for some integer $n>2$, then $P_2=\{-n,-(n-1),\ldots,-1\}$ is a continued fraction block, so $|B|=n-1>1$. Excluding the above cases, we may assume $q/p=[a_1,\ldots,a_m]$, where $m\geq2$ and $a_i\leq-2$ for $1\leq i\leq m$. When $a_m<-2$, we have $A=\{[a_1,\ldots,a_m],[a_1,\ldots,a_m+1],\ldots,[a_1,\ldots,a_{m-1}+1]\}$, and hence $|A|=|a_m+1|>1$. When $a_m=-2$ and $m=2$, we have $B=\{[a_1,-2],[a_1], \infty\}$, which has length $2$. When $a_m=-2$ and $m>2$, we have $B=\{[a_1, ..., a_{m-2}, a_{m-1}, -2],[a_1, ..., a_{m-2}, a_{m-1}],[a_1, ..., a_{m-2}], \ldots\}$, whose length is at least $2$. 
\end{proof}

The special case $q/p=-2$ will be treated separately in Example~\ref{example:(1,q)} and \ref{example:(1,q)transverse}, which provide a complete classification when $p=1$. Then, if $|A|=1$, we denote $P_1=(A_1,A_3,\ldots,A_{2n-1})$ and $P_2=(B_2,B_4, \ldots, B_{2m})$, where $A_1=A$; if $|B|=1$, we denote $P_1=(A_2,A_4,\ldots,A_{2n})$ and $P_2=(B_1,B_3, \ldots, B_{2m-1})$, where $B_1=B$.

According to the proofs of \cite[Proposition 3 and 7]{dg1}, when converting the first and second upper contact surgery components in Figure~\ref{Figure:tw=0} into contact $(-1)$-surgeries, the choice of stabilizations corresponds to the choice of signs on basic slices in $P_1$ and $P_2$, respectively. More precisely, suppose \[-\frac{q}{q'}=[x_1+1,x_2,\ldots,x_u], \text{ where } x_i\leq -2 \text{ for } 1\leq i\leq u,\]and
\[-\frac{q}{q-q'}=[y_1+1,y_2,\ldots,y_v],\text{ where } y_j\leq-2\text{ for }1\leq j \leq v.\]
When using Ding-Geiges algorithm in \cite{dg1} to convert the two contact surgeries above into contact $(-1)$-surgeries, only terms less than $-2$ produce the choice of stabilizations. We sequentially denote these terms as $x_{k_1},\ldots,x_{k_n}$ with $x_{k_i} < -2$ for $1\leq i\leq n$, and $y_{l_1},\ldots,y_{l_m}$ with $y_{l_j} < -2$ for $1\leq j\leq m$, respectively. Then, 
\begin{itemize}
    \item the number of positive stabilizations in the $|x_{k_i}+2|$ stabilizations equals to the number of positive basic slices in $A_{2i-1}$, for $1\leq i\leq n$;
    \item the number of positive stabilizations in the $|y_{l_j}+2|$ stabilizations equals to the number of negative basic slices in $B_{2j}$, for $1\leq j\leq m$.
\end{itemize}
In fact, we always combine those contact $(-1)$-surgeries on Legendrian push-offs of some component $L$ without stabilization into a single contact $(-1/k)$-surgery on $L$ for convenience. See Figure~\ref{Figure:p=1tw=0} for examples.

Now consider the concatenated path \(\overline{P_1} \cup P_2\), where \(\overline{P_1}\) is the reverse of \(P_1\). According to the method in \cite[Observation 2.13]{emm}, this path can be uniquely shortened to a one-edge path from \(\infty\) to \(-1\).



\begin{lemma}\label{lem:tw=0distinct}
    Given two pairs of decorated paths $(P_1,P_2)$ and $(P'_1,P'_2)$ that representing $q/p$. If $(P_1,P_2)\neq(P'_1,P'_2)$ then $L_{P_1,P_2}\neq L_{P'_1,P'_2}$; that is, $L_{P_1,P_2}$ and  $L_{P'_1,P'_2}$ are not coarsely equivalent. 
\end{lemma}

\begin{proof}
    Given a pair of decorated paths $(P_1,P_2)$ representing $q/p$. We perform a contact $(-\frac{1}{2})$-surgery on $L_{P_1,P_2}$ (that is, two Legendrian surgery along two copies of $L_{P_1,P_2}$) in Figure~\ref{Figure:tw=0}. We will convert the resulting surgery diagram back into the Farey graph and claim that the resulting contact manifold is a tight lens space. Firstly a contact $(-\frac{1}{2})$-surgery on $L_{P_1,P_2}$ will cancel the two surgery components with a $(+1)$-coefficient. Since the left contact surgeries both have negative coefficients, the resulting surgery diagram represents a tight contact structure. Secondly, the left components can be considered as the result of gluing a solid torus $S_1$ with a tight contact structure $\xi_1 \in \Tight(S_{-\frac{q}{q'}-1};-1)$ and another one $S_2$ with a tight contact structure $\xi_2 \in \Tight(S^{-\frac{q}{q''}-1};-1)$ along the lower and upper boundary, respectively, of an $I$-invariant thickened torus with the boundary slope $-1$. Set \[\psi=\begin{pmatrix}
        -q'' & q'\\-p''&p'
    \end{pmatrix}.\] Note that \[\begin{pmatrix}
        -q'' & q'\\-p''&p'
    \end{pmatrix}\begin{pmatrix}
        q+q'\\-q'
    \end{pmatrix}=\begin{pmatrix}
        -q^2\\-pq-1
    \end{pmatrix},\] and \[\begin{pmatrix}
        -q'' & q'\\-p''&p'
    \end{pmatrix}\begin{pmatrix}
        -q''\\q+q''
    \end{pmatrix}=\begin{pmatrix}
        q^2\\pq-1
    \end{pmatrix}.\] Then $\psi$ converts $(S_1,\xi_1)$ into an element in $\Tight(S_\frac{q^2}{pq+1};\frac{q}{p})$ and $(S_2,\xi_2)$ into an element in $\Tight(S^{\frac{q^2}{pq-1}};\frac{q}{p})$. The boundary slope of the thickened torus now becomes $\frac{q}{p}$. Therefore, the resulting manifold is the lens space $L^{\frac{q^2}{pq-1}}_{\frac{q^2}{pq+1}}$.

    Let $\tilde{P_1}$ and $\tilde{P_2}$ be the decorated paths that represent $(S_1,\xi_1)$ and $(S_2,\xi_2)$, respectively. By the algorithm to convert any contact surgery into a series of contact $(\pm1)$-surgeries (see \cite{dg1}), continued fraction blocks and the decoration of $P_1$ (respectively, $P_2$) are in one-to-one correspondence with those of $\tilde{P_1}$ (respectively, $\tilde{P_2}$). Thus $(P_1,P_2)\neq (P'_1,P'_2)$ implies that $(\tilde{P_1},\tilde{P_2})\neq (\tilde{P'_1},\tilde{P'_2})$.  Since $q'p-p'q=qp''-pq''=1$ and $q',q''>q$, we have \[\frac{q''}{p''}<\frac{q^2}{pq+1}<\frac{q}{p}<\frac{q^2}{pq-1}<\frac{q'}{p'},\] which means there is no chance for $\overline{\tilde{P_1}}\cup \tilde{P_2}$ and $\overline{\tilde{P'_1}}\cup \tilde{P'_2}$ to be shortened consistently at $\frac{q}{p}$. In other words, both $\overline{\tilde{P_1}}\cup \tilde{P_2}$ and $\overline{\tilde{P'_1}}\cup \tilde{P'_2}$ are shortest paths. Therefore, $\overline{\tilde{P_1}}\cup \tilde{P_2}$ and $\overline{\tilde{P'_1}}\cup \tilde{P'_2}$ are distinct decorated paths, up to shuffling within continued fraction blocks. By the classification of tight contact structures on lens spaces (see \cite{h}), these two paths correspond to non-isotopic tight contact structures on the lens space $L^{\frac{q^2}{pq-1}}_{\frac{q^2}{pq+1}}$. However, if $L_{P_1,P_2}$ is coarsely equivalent to $L_{P'_1,P'_2}$ then their complements are related by a contactomorphism which is smoothly isotopic to the identity. Performing contact $(-\frac{1}{2})$-surgery on $L_{P_1,P_2}$ and $L_{P'_1,P'_2}$ yields two tight contact structures on $L^{\frac{q^2}{pq-1}}_{\frac{q^2}{pq+1}}$, which extend uniquely (in a tight way) from the complement. These two structures must therefore be related by a contactomorphism smoothly isotopic to the identity and connected by a one‑parameter family of contact structures; hence they are contact isotopic. This yields a contradiction.
\end{proof}


Let $(P_1,P_2)$ be a decorated pair of paths and subdivide $P_1$ and $P_2$ into continued fraction blocks. Suppose $2\leq k\in\N$. Following \cite[Section 2.3]{emm}, we call $(P_1,P_2)$ \dfn{$k$-consistent} if the signs of basic slices in $A_i$ and $B_i$ for $i\leq k$ are all the same. We call $(P_1,P_2)$ \dfn{$k$-inconsistent} if $(P_1,P_2)$ is $(k-1)$-consistent but not $k$-consistent. If all basic slices in $A_i$ for $i\leq k$ have the same sign and all basic slices in $B_i$ for $i\leq k$ have the opposite sign, then $(P_1,P_2)$ is called \dfn{totally $k$-inconsistent}. Moreover, we can convert a $k$-inconsistent pair of paths into a unique $(k-1)$-inconsistent pair of paths, and hence into an $i$-inconsistent one for all $i=2,3,\ldots,k$. All these decorated pairs of paths define the same contact structure on $S^1\times S^2$ and are said to be \dfn{compatible}.

\subsection{Surgery on torus knots in \texorpdfstring{$S^1\times S^2$}{S1 x S2}.}
\begin{lemma}\label{lemma:surgerychar}
The $r\lambda+s\mu$ surgery along the torus knot $T_{p,q}$ in $S^1\times S^2$ is diffeomorphic to the closed Seifert fibered space $M(S^2; \frac{q'}{q}, -\frac{q'}{q}, -\frac{r}{s})$ for $s\neq0$, and $L(q,p)\#L(-q,p)$ for $s=0$.
\end{lemma}

\begin{proof}
The complement $C$ of $T_{p,q}$ in $S^1\times S^2$ is diffeomorphic to $M(D^2; \frac{q'}{q}, -\frac{q'}{q})$. It can be decomposed into $C=V_{1}\cup(S^1\times P)\cup V_2$, where $P$ is a pair of pants, $\partial(S^1\times P)=T_{1}\cup T_{2}\cup T_{3}$, $T_{i}$ is identified with $\partial V_i$ for $i=1,2$ and $T_3=\partial C$. By \cite[Page 40]{Ha}, $M(S^2; \frac{q'}{q}, -\frac{q'}{q})$ is diffeomorphic to $S^1\times S^2$, and $T_3\cap \{\theta\}\times P$ can be identified with the meridian $\mu$ of the knot $T_{p,q}$. Recall that $\lambda$ corresponds to the regular fiber, i.e., $S^{1}\times\{p\}$ in $T_3$. Since $-T_3=\partial N$, the $r\lambda+s\mu$ curve in $\partial N$ is identified to a $-\frac{s}{r}$ curve on $T_3$. 

If $s\neq0$, then the lemma follows from the construction of Seifert fibered space in \cite[Section 2.1]{Ha}. If $s=0$, then it is easy to see that there is a separating sphere which cuts the surgered manifold into two punctured lens spaces $L(q,q')$ and $L(q,-q')$ which are $L(q,p)$ and $L(-q,p)$, respectively.  
\end{proof}


\subsection{Legendrian torus knots in tight \texorpdfstring{$S^1\times S^2$}{S1 x S2}.}\label{subsection:tight} The standard tight contact structure $\xi_{std}$ on $S^1\times S^2$ is defined as the kernel of $(x_{3}d\theta+x_{1}dx_{2}-x_{2}dx_{1})|_{S^1\times S^2}.$ 
 Let $r_{\theta}$ be the rotation matrix of $\mathbb{R}^3$ about the $x_3$-axis:
$$\begin{pmatrix} 
\cos\theta & \sin\theta & 0 \\
-\sin\theta & \cos\theta & 0 \\
0 & 0 & 1
\end{pmatrix}$$
Then \begin{align*}
  \delta\colon S^1 \times S^2 &\to S^1 \times S^2,\\ 
  (\theta,x_1, x_2, x_3) &\mapsto (\theta, r_\theta(x_1, x_2, x_3)).
\end{align*} is an orientation-preserving diffeomorphism.  Let $Cont_{0}(S^1\times S^2, \xi_{std})$ denote the set of contactomorphisms of $(S^1\times S^2, \xi_{std})$ which are smoothly isotopic to the identity.  By \cite{dg} and \cite{m}, $\delta^2$ generates $$\pi_{0}(Cont_{0}(S^1\times S^2, \xi_{std}))\cong\mathbb{Z}.$$ 

We explain Remark~\ref{rmk:tw=0intight}.  The Euler class of the tight contact structure on $S^1\times S^2$ is $0$. So one can define rotation number $\rot(L)$ for any Legendrian knot $L$ in $(S^1\times S^2, \xi_{std})$. According to \cite[Proof of Theorem 1.1]{cdl}, $\rot(\delta^{2}(L))=\rot(L)+2q$ for a Legendrian torus knot $T_{p,q}$ in $(S^1\times S^2, \xi_{std})$. By \cite[Proposition 1.4]{cdl}, there are exactly $2$ Legendrian $(p,q)$-torus knots with $\tw=0$ in $(S^1\times S^2, \xi_{std})$ up to coarse equivalence. The rotation numbers of these two Legendrian knots are $p \pmod {2q}$ and $-p\pmod {2q}$, respectively.

\section{Tight contact structures on torus knot complements in \texorpdfstring{$S^1\times S^2$}{S1 x S2}}
Let $N$ denote a neighborhood of the torus knot $T_{p,q}$ in $S^1\times S^2$, and let $C$ be the closure of $S^1\times S^2\setminus N$. Then $C$ is the Seifert fibered space $M(D^2; \frac{q'}{q}, -\frac{q'}{q})$ \cite{Ha}. We fix a framing convention for the torus $-\partial C=\partial N$. Let $\lambda$ be the framing induced by the Heegaard torus containing $T_{p,q}$, and $\mu$ the meridian of $T_{p,q}$. A simple closed curve on $\partial N$ homologous to $a\lambda+b\mu$ is said to have slope $b/a$. Note that $\lambda$ is a regular fiber of $C$. A regular fiber of $C$ is also referred to a {\it vertical circle}. The {\it twisting number} of a Legendrian vertical circle is defined as the difference of its contact framing and the framing induced by the Seifert fibration.

Now we build two different topological models of $C$. Consider $S^1\times S^2=\{(\theta, x_1, x_2, x_3)\in S^1\times \mathbb{R}^3\mid \theta\in S^1, x_1^2+x_2^2+x_3^2=1\}$. The Heegaard torus $T_0 = \{(\theta, x_1, x_2, x_3) \in S^1 \times S^2 \mid x_3 = 0\}$. Then $S^1\times S^2\setminus T_0$ is the union of two solid tori: $V_1=\{(\theta, x_1, x_2, x_3) \in S^1 \times S^2 \mid x_3 \leq 0\}$ and $V_2=\{(\theta, x_1, x_2, x_3) \in S^1 \times S^2 \mid x_3 \geq 0\}$. Slightly thickening $T_0$, we can view $S^1\times S^2=V_1\cup(T_0\times [0,1])\cup V_2$. Here we use the coordinate system coming from the Seifert framing of $V_1$, so that the meridians of $V_1$ and $V_2$ both have slope $0$. Take the neighborhood $N$ of $T_{p,q}$ to be contained in the interior of $T_0\times [0,1]$. In $T_0\times [0,1]\setminus N$, we can find an annulus $A$, whose boundary components are two $(p,q)$-curves on $V_1$ and $V_2$. Hence we have a model $C=V_1\cup N(A) \cup V_2$, where $N(A)$ is a neighborhood of $A$ in $C$.

We can also view $T_0\times [0,1]\setminus N$ as $S^1\times P$, where $P$ is a pair of pants. Then we have another model $C=V_1\cup (S^1\times P)\cup V_2$. Denote $\bdy (S^1\times P)=T_1\cup T_2\cup T_3$, where $T_i$ is identified with $\bdy V_i$ for $i=1,2$. Take coordinates on each $T_i$ so that $S^1\times \{\text{pt}\}$ has slope $0$ and $(\{\theta\}\times P)\cap T_i$ has slope $\infty$, for $i=1,2,3$. Then we can convert the first coordinate system to the second by the map \[\phi=\begin{pmatrix}
        p&-q\\-p'&q'
    \end{pmatrix}.\]

Given the complement $C$ of a $(p, q)$-torus knot and a slope $s$ on the boundary of C, denote $\Tight_t(C; s)=$\{Tight contact structures in $C$ up to isotopy, with convex boundary having two dividing curves of slope $s$ and convex $l$ Giroux torsion for $t\in\frac{1}{2}\N\cup\{0\}$\}. 

\begin{lemma}\label{lemma:s>0}
    Suppose $s>0$. For any tight contact structure $\xi\in \Tight_0(C;s)$, there is a contact structure $\xi'\in\Tight_0(C;\infty)$ and a contact structure $\xi''\in \Tight^{min}(T^2\times[0,1];s,\infty)$ such that $\xi$ is isotopic to $\xi'\cup \xi''$ under the natural identification $C\cong C\cup (T^2\times[0,1])$.
\end{lemma}
\begin{proof}
    We proceed via a strategy analogous to that in \cite[Lemma 6.2]{emm}. Suppose that the complement $C$ contains a vertical Legendrian circle $L$ with twisting number $0$. Then there exists a torus $T$, parallel to $\partial C$, that contains $L$. Perturb $T$ to be convex, then the dividing slope of $T$ is $0$. Since $C$ has positive boundary slope, there also exists a convex torus $T'$, parallel to $\partial C$, with dividing slope $\infty$. This torus $T'$ thus yields the desired splitting of $C$.

    Now we assume that $C$ contains no $0$-twisting vertical Legendrian curves. Consider the model for the complement $C$ as $V_1\cup N(A) \cup V_2$. We perturb $\partial V_i$ ($i=1,2$) to be convex and let $\frac{n_i}{m_i}$ be the boundary slope of $V_i$, where $m_i\geq0$. By the assumption, the slope $\frac{n_1}{m_1}\in (0,\infty)\cup (\infty, \frac{q}{p})$ and $\frac{n_2}{m_2}\in (\frac{q}{p},0)$. We can arrange both ruling slopes of $\partial V_i$ are $\frac{q}{p}$ and the annulus $A$ to be convex whose boundary are two ruling curves on $\partial V_i$. Now we start with $V_i$ as a standard neighborhood of a Legendrian knot with a very negative twisting. If two boundary components of $A$ have different twisting numbers then by Imbalance Principle there exists a bypass in $A$. After a bypass attachment, $V_1$ or $V_2$ will thicken. Iterating this process until we have a final 'stage' where
    \begin{align}\label{eq0}
        |\frac{q}{p}\bigcdot \frac{n_1}{m_1}|=|\frac{q}{p}\bigcdot \frac{n_2}{m_2}|
    \end{align}
    
    and all the dividing curves on $A$ run from one boundary component to the other. Since $\frac{n_2}{m_2}\in  (\frac{q}{p},0)$,  if $\frac{n_1}{m_1}\in (0,\infty)$, then Equation~\ref{eq0} becomes
      \begin{align}\label{eq1}
    \frac{n_1}{m_1}\bigcdot \frac{q}{p}=\frac{n_2}{m_2}\bigcdot \frac{q}{p};
    \end{align}
    if  $\frac{n_1}{m_1}\in (\infty, \frac{q}{p})$, then Equation~\ref{eq0} becomes \begin{align}\label{eq2}
    \frac{q}{p}\bigcdot \frac{n_1}{m_1}=\frac{n_2}{m_2}\bigcdot\frac{q}{p}.
    \end{align}

    Since the slope $\frac{n_1}{m_1}$ is obtained from the slope of $\partial V_1$ after some amount of bypass attachments along $q/p$-sloped curves, the slope $\frac{n_1}{m_1}$ must be one of the vertices in the shortest path from the slope of $\partial V_1$ clockwise to $q/p$. That is to say the possible slope of $\partial V_1$ is an element in one of the following sets:
    \begin{itemize}
        \item $\mathcal{S}_1$ consisting of slopes in the form of $\frac{1}{m_1}$ with $m_1\geq 0$;
        \item $\mathcal{T}_1$ consisting of vertices from $\frac{q}{p}$ anti-clockwise to $\lfloor \frac{q}{p} \rfloor$ in the path $P_1$ on the Farey graph.
    \end{itemize}
    
    Similarly, the possible boundary slope of $V_2$ is an element in one of the following sets:
    \begin{itemize}
        \item $\mathcal{S}_2$ consisting of slopes in the form of $\frac{-1}{m_2}$ with $m_2\geq 1$;
        \item $\mathcal{T}_2$ consisting of vertices from $\frac{q}{p}$ clockwise to $-1$ in the path $P_2$ on the Farey graph.
    \end{itemize}

    {\textbf{Case 1:} $\frac{n_1}{m_1}\in \mathcal{S}_1$ \text{and} $\frac{n_2}{m_2}\in \mathcal{S}_2$.}
    
    Equation~\ref{eq1} now becomes $p-qm_1=-p-qm_2$. Then we have that $q(m_1-m_2)=2p$. Thus $q=-1$ or $q=-2$. Since $-q>p\geq 1$ we have $(p,q)$ must be $(1,-2)$ and all possible solutions in this case must be $(\frac{n_1}{m_1},\frac{n_2}{m_2})=(\frac{1}{m_1},\frac{-1}{m_2})$ where $m_2=m_1+1$. 
    
    We claim that any such solution results in a torus knot complement $C$ with $\infty$ boundary slope (with respect to the torus framing). In fact, we consider a diffeomorphism $\phi\in SL(2;\mathbb{Z})$ to change the coordinate system. Set\[\phi=
    \begin{pmatrix}
        1&2\\-1&-1
    \end{pmatrix}.\] Clearly, the map $\phi$ changes slope $\frac{q'}{p'}=\frac{-1}{1}$ into $\frac{1}{0}$, $\frac{q}{p}=\frac{-2}{1}$ into $\frac{0}{1}$ and $\frac{q''}{p''}=\frac{-1}{0}$ into $\frac{-1}{1}$. Especiallly, we have that\[\frac{1}{m_1}\xrightarrow{\phi}-\frac{1+2m_1}{1+m_1}=-\frac{2m_2-1}{m_2},\] and \[\frac{-1}{m_2}\xrightarrow{\phi}-\frac{2m_2-1}{m_2-1}.\]
    Rounding edges in the model $C\cong V_1 \cup (S^1\times P)\cup V_2$, we obtain the boundary slope of $C$ as \[s(\partial C)=\frac{2m_2-1}{-(m_2-1)+m_2-1}=\infty.\] 
    
    {\textbf{Case 2:} $\frac{n_1}{m_1}\in \mathcal{S}_1$ \text{and} $\frac{n_2}{m_2}\in \mathcal{T}_2$.}

    (1) Suppose that $\frac{n_2}{m_2}$ is an integer  between $-1$ and $\lceil \frac{q}{p} \rceil$. Then $m_2=1$ and $n_2<0$, and Equation~
    \ref{eq1} becomes $p-qm_1=pn_2-q$. We obtain $m_1=pK+1$ and $n_2=1-qK$ for some $K\in\mathbb{Z}$. If $m_1=0$, then necessarily $p=1$ and $K=-1$. This yields a unique solution in the case when $p=1$, namely $(\frac{n_1}{m_1},\frac{n_2}{m_2})=(\infty,q+1)$, which will be addressed in Case 4. If $m_1>0$, then we must have $K\geq 0$, while $n_2<0$ implies $K<0$-a contradiction. 
    
    (2) Otherwise, suppose $\frac{n_2}{m_2}$ is a vertex of $P_2$ lying between $\lceil \frac{q}{p}\rceil$ and $\frac{q}{p}$. Then  \[1=\frac{q'}{p'}\bigcdot \frac{q}{p}\leq \frac{n_2}{m_2}\bigcdot \frac{q}{p}<\cdots\leq\lceil\frac{q}{p}\rceil\bigcdot\frac{q}{p}=p\lceil \frac{q}{p}\rceil-q<p.\] On the other hand \[\frac{1}{m_1}\bigcdot \frac{q}{p}=p-qm_1\geq p.\] This contradicts to Equation~\ref{eq1}.
    
    Therefore, there is no solution for $(\frac{n_1}{m_1},\frac{n_2}{m_2})$ in this case.

    {\textbf{Case 3:} $\frac{n_1}{m_1}\in \mathcal{T}_1$ \text{and} $\frac{n_2}{m_2}\in \mathcal{S}_2$.}
    
    Now we have $n_2=-1$, and Equation~\ref{eq2} becomes $q(m_1+m_2)=p(n_1-1)$.  Since $n_1<0$, Equation~\ref{eq2} implies that $n_1=Kq+1$ for some positive integer $K$. Because $\frac{n_1}{m_1}$ is a vertex in $P_1$ lying between $\frac{q}{p}$ and $\lfloor \frac{q}{p} \rfloor$, Lemma~\ref{lemma:paths} gives $q\leq n_1\leq \lfloor \frac{q}{p} \rfloor$, which forces $K=1$. Thus $n_1=q+1$ and $m_1+m_2=p$. So $\frac{n_1}{m_1}=\frac{q+1}{p-m_2}$. By Lemma~\ref{lemma:paths}, we have $\frac{n_1}{m_1}=\frac{q''}{p''}$, which leads to $q(p-m_2)-p(q+1)=1$. 
    Solving this yields $m_2=1$ and $q=-p-1$. Hence, a solution exists only when  $(p,q)=(p,-p-1)$, in which case Equation~\ref{eq2} is solved by $(\frac{n_1}{m_1},\frac{n_2}{m_2})=(\frac{-p}{p-1}, \frac{-1}{1})$. This situation will be addressed in Case 4.
    
    {\textbf{Case 4:} $\frac{n_1}{m_1}\in \mathcal{T}_1$ \text{and} $\frac{n_2}{m_2}\in \mathcal{T}_2$.}

    In this case, by Lemma~\ref{lemma:paths}, we have $$q''\leq n_1\leq  \lfloor \frac{q}{p} \rfloor ~\text{and}~ q'\leq n_2\leq  -1.$$ It follows that $$q=q''+q'\leq n_1+ n_2\leq \lfloor \frac{q}{p} \rfloor-1<0.$$ Note that Equation~\ref{eq2} is equivalent to $q(m_1+m_2)=p(n_1+n_2)$ which implies that $q\mid n_1+n_2$. So $n_1=q''$ and $n_2=q'$. Therefore Equation~\ref{eq2} has a unique solution in this case: $(\frac{n_1}{m_1},\frac{n_2}{m_2})=(\frac{q''}{p''},\frac{q'}{p'})$. As in Case 1, we set \[\phi=\begin{pmatrix}
        p&-q\\-p'&q'
    \end{pmatrix}.\] Then $\phi$ sends $\frac{q'}{p'}$ to $\infty$, $\frac{q}{p}$ to $0$ and $\frac{q''}{p''}$ to $-1$. Rounding edges in the model $C\cong V_1 \cup (S^1\times P) \cup V_2$, we obtain the boundary slope of $C$ as \[s(\partial C)=\frac{1}{0+1-1}=\infty.\]

    In summary, the complement $C$ contains a convex torus with two dividing curves of slope $\infty$ in every case, thereby yielding the claimed splitting.
\end{proof}

\begin{lemma}\label{lemma:destabilization}
If $L$ is a non-loose Legendrian torus knot $T_{p,q}$ in an overtwisted $S^1\times S^2$ with $\tw(L)<0$, then $L$ destabilizes.    
\end{lemma}

\begin{proof}
The tight contact structure on the complement $C$ of a standard neighborhood of $L$ belongs to $\Tight_i(C;\tw(L))$, where $i$ is the amount of Giroux torsion. If $i>0$, then there exists a convex torus $T$ paralleled to $\partial C$ with two dividing curves of slope $\tw(L)+1$, and the basic slice co-bounded by $T$ and $\partial C$ yields a destabilization of $L$. Now suppose $i=0$. As in the proof of Lemma~\ref{lemma:s>0}, if there exists a $0$-twisting vertical Legendrian curve, then we can find a convex torus with two dividing curves of slope $0$, and subsequently a convex torus $T$ with dividing slope $\tw(L)+1$. The basic slice bounded by this torus and $\partial C$ then gives a destabilization of $L$.  If no such $0$-twisting vertical Legendrian curve exists, we repeat the procedure from the proof of Lemma~\ref{lemma:s>0} to obtain a convex torus in $C$ parallel to $\partial C$ with slope $\infty$, and hence obtain a convex torus in $C$ parallel to $\partial C$ with slope $0$. This contradicts with the nonexistence of $0$-twisting vertical Legendrian curve.
\end{proof}

\begin{lemma}\label{lemma:tw=0}
    Up to coarse equivalence, the number of Legendrian $(p, q)$-torus knots with $\tw=0$, $\tor=0$ and tight complement in some contact structure on $S^1\times S^2$ is exactly $m(p, q)$.
\end{lemma}

\begin{proof}
We first derive an upper bound through a discussion of tight contact structures on the complement. In the model $C = V_1 \cup (S^1 \times P) \cup V_2$, we set $\partial(S^1 \times P) = T_1 \cup T_2 \cup T_3$ and adopt the coordinate system on the $T_i$. For $i = 1, 2$, we take convex annuli in $S^1 \times P$ that connect $T_3 = \partial C$ to $T_i = \partial V_i$, and whose boundaries are isotopic to $S^1 \times \{\text{pt}\} \subset S^1 \times P$. Because the dividing curves on $\partial C = T_3$ have slope 0, the Imbalance Principle guarantees the existence of bypasses on these annuli, allowing us to thicken $V_i$ until its boundary slope becomes 0. Note that by \cite[Lemma 2.21]{emm}, there is a unique isotopy class of contact structures on $S^1 \times P$ when all boundary slopes are 0. It follows that the contact structure on $C$ is uniquely determined, up to isotopy, by the contact structures on $V_1$ and $V_2$.
    
Back to the original model $S^1\times S^2=V_1\cup (T_0\times [0,1])\cup V_2$, the slope $0$ in the framing on $T_1$ corresponds to slope $q/p$ in the framing of $\bdy V_1$. By Lemma~\ref{lemma:upper0} and Lemma~\ref{lemma:lower0}, we have $$|\Tight(V_1)|=|\Tight(S_0; \frac{q}{p})|=|(b_{1}+1)\cdots(b_{n-1}+1)b_{n}|$$ and $$|\Tight(V_2)|=|\Tight(S^0; \frac{q}{p})| =|(a_{1}+1)\cdots(a_{m-1}+1)a_{m}|.$$ Hence $$m(p, q)=|(a_{1}+1)\cdots(a_{m-1}+1)a_{m}|\cdot |(b_{1}+1)\cdots(b_{n-1}+1)b_{n}|$$ is an upper bound.

We now establish a lower bound by constructing the desired Legendrian $(p, q)$-torus knots. Given a path $P_1$ representing an element of $\Tight(S_0; q/p)$ and a path $P_2$ representing an element of $\Tight(S^0; q/p)$, the pair $(P_1, P_2)$ defines a contact structure $\xi_{P_1,P_2}$ on $S^1\times S^2$. Let $L_{P_1,P_2}$ be a Legendrian divide on $T_0$. By Lemma~\ref{lem:divide}, any such $L_{P_1,P_2}$ can be represented by a contact surgery diagram as shown in Figure~\ref{Figure:tw=0}.

    A contact $(-1)$-surgery along $L_{P_1,P_2}$ cancels the $(+1)$-surgery on the parallel knot, yielding a tight contact structure. Thus, the complement of $L_{P_1,P_2}$ is tight and torsion-free. Moreover, by Lemma~\ref{lem:tw=0distinct}, distinct pairs $(P_1, P_2) \neq (P_1', P_2')$ yield knots $L_{P_1,P_2}$ and $L_{P_1',P_2'}$ that are not coarsely equivalent. Hence, $m(p, q)$ serves as a lower bound, and the proof is finished.
    \end{proof}

\begin{lemma}\label{lemma:s=infty}
$|\Tight_0(C;\infty)|=n(p,q)$.
\end{lemma}

\begin{proof}
    In the model $C = V_1 \cup N(A) \cup V_2$, where $N(A) = A \times [-1,1]$ is a neighborhood of the annulus $A$, consider any contact structure $\xi \in \Tight_0(C; \infty)$. We first show that $(C, \xi)$ cannot contain a vertical Legendrian curve with twisting number $0$. Suppose, to the contrary, that such a curve exists. Then there is a copy $C' \subset C$ such that $C \setminus C' \cong T^2 \times [0,1]$, with $\xi|_{C'} \in \Tight_0(C; 0)$ and $\xi|_{T^2 \times [0,1]} \in \Tight^{\text{min}}(T^2 \times [0,1]; \infty, 0)$. If every convex torus in $C'$ parallel to the boundary has dividing slope $0$, then by Lemma~\ref{lemma:2-inconsistent} and the proof of Lemma~\ref{lemma:tw=0}, the structure $\xi|_{C'} \in \Tight_0(C; 0)$ can be described by a $2$-consistent pair of paths $(P_1, P_2)$. It then follows from Lemma~\ref{lemma:overtwisted} that $\xi$ is overtwisted—a contradiction to the assumption that $\xi \in \Tight_0(C; \infty)$. If, instead, $\xi|_{C'}$ contains a convex torus parallel to the boundary with a nonzero dividing slope, then by Lemma~\ref{lemma:s>0}, it contains a further copy $C'' \subset C'$ such that $\partial C''$ has dividing slope $\infty$. In this case, $C \setminus C''$ forms a convex Giroux torsion layer, again contradicting the hypothesis that $\xi \in \Tight_0(C; \infty)$.

    Since $(C,\xi)$ contains no  $0$-twisting vertical Legendrian curve, the proof of Lemma~\ref{lemma:s>0} implies only two possibilities: either $(p,q) = (1,-2)$, or we have $\xi|_{V_1} \in \Tight(V_1, (q/p)^a)$, $\xi|_{V_2} \in \Tight(V_2, (q/p)^c)$, and $N(A)$ is an $I$-invariant neighborhood of a convex annulus $A$ with two dividing curves. We claim that the first case is contained in the second.


    We now prove the claim. Consider the model \( C = V_1 \cup (S^1 \times P) \cup V_2 \), where \( V_1 \) has boundary slope \( -\frac{2m_2 - 1}{m_2} \) and \( V_2 \) has boundary slope \( -\frac{2m_2 -  1}{m_2 - 1} \), with \( m_2 \geq 1 \). Let \( A_1, A_2 \) be convex annuli in \( S^1 \times P \) connecting \( \partial V_1 \) and \( \partial V_2 \), respectively, to \( \partial C \), such that their boundaries are ruling curves isotopic to \( S^1 \times \{\text{pt}\} \).  Observe  that $$|\frac{0}{1} \bigcdot (-\frac{2m_2-1}{m_2})|=2m_2-1\geq|\frac{0}{1}\bigcdot \frac{1}{0}|=1$$and $$|\frac{0}{1} \bigcdot (-\frac{2m_2-1}{m_2-1})|=2m_2-1\geq|\frac{0}{1}\bigcdot \frac{1}{0}|=1.$$ By the Imbalance Principle, bypasses exist that allow thickening of \( V_1 \) and \( V_2 \), unless \( m_2 = 1 \). In the case $m_2=1$, $V_1$ has boundary slope $\frac{1}{m_2-1}=\infty=(\frac{q}{p})^a$, and $V_2$ has boundary slope $\frac{-1}{m_2}=\frac{-1}{1}=(\frac{q}{p})^c$. In this situation, when rounding edges, the boundary slope of \( C' = V_1 \cup N(A) \cup V_2 \) becomes \( \infty \), and hence \( C \setminus C' \) is \( I \)-invariant.

    
    Thus $|\Tight_0(C;\infty)|$ is bounded above by the product of the number of tight contact structures on $V_1$ with boundary slope $(q/p)^a$ and the number of tight contact structures on $V_2$ with boundary slope $(q/p)^c$. By Lemma~\ref{lemma:L(q,p)}, we have
$$|\Tight(V_1; (\frac{q}{p})^a)|\cdot|\Tight(V_2; (\frac{q}{p})^c)|=|\Tight(L(q,p)\#L(-q,p))|=n(p,q).$$
    Hence, $n(p,q)\geq |\Tight_0(C;\infty)|$.

    By Lemma~\ref{lemma:surgerychar}, the manifold $L(q,p) \# L(-q,p)$ can always be decomposed as the union of $C$ and a solid torus $S^1 \times D^2$, where the meridian $\{\theta\} \times \partial D^2$ is identified with a curve of slope $0$ on $\partial C$ (with respect to the torus framing). The tight contact structure on $L(q,p) \# L(-q,p)$, when restricted to $C$, has boundary slope $\infty$ and contains no convex Giroux torsion. Indeed, if it did contain such torsion, it would produce a solid torus with a convex torsion layer, which in turn contains an overtwisted disk—a contradiction. Thus, we have 
    $$n(p,q)\leq |\Tight_0(C;\infty)|\cdot |\Tight(S^1\times D^2; 0)|=|\Tight_0(C;\infty)|.$$
Therefore, we have $|\Tight_0(C;\infty)|=n(p,q) $.\end{proof}

\begin{lemma}\label{lemma:s=n>0}
For any positive integer $i$,  $|\Tight_0(C;i)|=2n(p,q)$.
\end{lemma}

\begin{proof}
    By Lemma~\ref{lemma:s>0}, we know that for any integer $i > 0$ and $\xi\in \Tight_0(C; i)$ there is a copy $C'$ inside $C$, such that $\xi|_{C'}\in \Tight_0(C;\infty)$, $C\setminus C'=T^2\times [0,1]$ and $\xi|_{T^2\times [0,1]}\in \Tight^{min}(T^2\times [0,1];i,\infty)$. Then $$|\Tight_0(C;i)|\leq |\Tight_0(C;\infty)|\cdot|\Tight^{min}(T^2\times [0,1];i,\infty)|=2n(p,q).$$
    To establish the equality of both sides, it suffices to show that the following map is injective: 
    \[
        \Tight_0(C;\infty)\times \Tight^{min}(T^2\times [1,2]; i,\infty) \to \Tight_0(C; i),\, (\eta,\zeta)\mapsto \eta\cup\zeta.
    \]
    We first verify that this map is well-defined. Let $\zeta\in \Tight^{min}(T^2\times [1,2]; i,\infty)$ be the positive basic slice (the negative case being analogous). In the proof of Lemma~\ref{lemma:s=infty}, when gluing a solid torus $S=S^1\times D^2$, with meridional slope $0$ and the unique tight structure, to $(C,\eta)$, we obtain a tight contact structure on $L(q,p)\#L(-q,p)$. View $S$ as a standard neighborhood of some Legendrian knot $L$ in $L(q,p)\#L(-q,p)$. We positively stabilize $L$ once, and obtain a standard neighborhood $S'$ of $S_{+}(L)$, where $S\setminus S'$ is a positive basic slice $\tilde{\zeta}\in \Tight^{min}(T^2\times [0,2]; 1,\infty)$. We may assume that the convex torus $T^2\times \{1\}\subset S\setminus S'$ has dividing slope $i$. Then, $\zeta$ is just the restriction of $\tilde{\zeta}$ on $T^2\times [1,2]$.  Therefore, $\eta\cup \zeta$ is contained in a tight contact structure on $L(q,p)\#L(-q,p)$, and hence tight. Next, if $\eta \cup \zeta$ has convex Giroux torsion, we can thicken $S$ to contain convex Giroux torsion, which would produce an overtwisted disk in $L(q,p)\#L(-q,p)$. Thus $\eta \cup \zeta \in \Tight_0(C; i)$.

    Now we prove the injectivity. Let $\zeta$ and $\zeta'$ be the positive and negative basic slices in $\Tight^{min}(T^2\times [1,2]; i,\infty)$, respectively. We will show that $\eta\cup\zeta$ is not isotopic to $\eta'\cup\zeta'$ for any $\eta$ and $\eta'$ in $\Tight_0(C;\infty)$. 
    
    For $i>1$, attach a solid torus $S''$ to $C\cup(T^2\times [1,2])$ with meridional slope $0$, and equip $S''$ with a tight contact structure $\xi$ for which all basic slices are positive. On the thickened solid torus $(T^2\times [1,2])\cup S''$,  $\zeta\cup\xi$ is tight, whereas $\zeta'\cup\xi$ becomes overtwisted. In fact, $\zeta\cup\xi$ is the unique contact structure on $S$. Hence, by the proof of Lemma~\ref{lemma:s=infty}, $\eta\cup\zeta\cup\xi$ is a tight contact structure on $C\cup (T^2\times [1,2])\cup S''=L(q,p)\#L(-q,p)$. However, $\eta'\cup\zeta'\cup\xi$ is overtwisted on $L(q,p)\#L(-q,p)$. Consequently, $\eta\cup\zeta$ and $\eta'\cup \zeta'$ are not isotopic.
    
    When $i=1$, we again glue the solid torus $S''$ to $C\cup (T^2\times [1,2])$ with meridional slope $0$, and take the unique tight contact structure $\xi$ on $S''$. In this case, on the thickened solid torus $(T^2\times [1,2])\cup S''$, both $\zeta\cup\xi$ and $\zeta'\cup\xi$ are isotopic to the unique tight contact structure on $S$. By the proof of Lemma~\ref{lemma:s=infty}, both $\eta\cup\zeta\cup\xi$ and $\eta'\cup\zeta'\cup\xi$ are tight contact structures on $C\cup S=L(q,p)\#L(-q,p)$. Moreover, if $\eta$ and $\eta'$ are not isotopic, then $\eta\cup\zeta\cup\xi$ and $\eta'\cup\zeta'\cup\xi$ are also non-isotopic. It follows that $\eta\cup\zeta$ and $\eta'\cup \zeta'$ are not isotopic. On the other hand, $\eta\cup\zeta$ and $\eta\cup \zeta'$ are not isotopic. This follows as the relative Euler classes evaluate differently on a horizontal surface in $C\cup (T^2\times [1,2])$ which has genus $0$ and has $-q$ boundary components with slope $\infty$.

    It remains to show that if $\eta\cup \zeta$ and $\eta'\cup \zeta$ are isotopic then $\eta$ and $\eta'$ are isotopic (the case for $\zeta'$ being analogous). Gluing a solid torus $S''$ as above, we will obtain the same contact structure on $L(q,p)\#L(-q,p)$. When decomposing it into $(C;\infty)$ and a solid torus $S$ as the proof of Lemma~\ref{lemma:s=infty}, we conclude that $\eta$ is isotopic to $\eta'$.
\end{proof}


\begin{lemma}\label{lemma:2-inconsistent}
    For any pair of paths $(P_1, P_2)$ representing $q/p$, denote $\xi'_{P_1,P_2}\in \Tight_0(C;0)$ the restriction of $\xi_{P_1,P_2}$ on the complement of $L_{P_1,P_2}$. Then, all convex tori parallel to $\bdy C$ have slope $0$ in $(C,\xi'_{P_1,P_2})$ if and only if $(P_1,P_2)$ is $2$-consistent. Moreover, if $(P_1, P_2)$ is $2$-inconsistent, then there is a copy $C'$ inside $C$ such that $(\xi'_{P_1,P_2})|_{C'}\in \Tight_0(C;\infty)$ and $(\xi'_{P_1,P_2})|_{C\setminus C'}\in \Tight^{min}(T^2\times [0,1];0,\infty)$.
\end{lemma}

\begin{proof}
    We first show that if $P_1$ and $P_2$ are $2$-inconsistent, then there exists a copy $C' \subset C$ such that $(\xi'_{P_1,P_2})|_{C'} \in \Tight_0(C;\infty)$. Recall from Section~\ref{section:paths} that the contact structure $\xi_{P_1,P_2}$ is constructed using solid tori $V_1$ and $V_2$ with convex boundaries $T_1$ and $T_2$, each of dividing slope $q/p$. Inside $V_1$ and $V_2$, we choose convex tori $T_1'$ and $T_2'$ with dividing slopes $(q/p)^a$ and $(q/p)^c$, respectively. Then,  in the model $C = V_1 \cup (S^1 \times P) \cup V_2$, the torus $T_1'$ has dividing slope $-1$, $T_2'$ has dividing slope $\infty$, and each $T_i$ (for $i=1,2,3$) has dividing slope $0$. The thickened tori $N_i$ cobounded by $T_i$ and $T_i'$ form basic slices for $i=1,2$. Since $P_1$ and $P_2$ are $2$-inconsistent, we may arrange their first basic slices to carry opposite signs. Note that reversing the orientation of a thickened torus changes its sign. Therefore, within $S^1 \times P$, the basic slices $N_1$ and $N_2$ actually carry the same sign. By \cite[Lemma 2.22]{emm}, there exists a convex annulus $A$ whose boundary consists of $0$-sloped ruling curves on $T_1'$ and $T_2'$, and which contains exactly two dividing curves running from one boundary component to the other. Rounding the edges of $T_1' \cup N(A) \cup T_2'$ yields a convex torus parallel to $\bdy C$ with dividing slope $\frac{1}{0+1-1}=\infty.$

    Next, we prove that if there exists a convex torus $T$ parallel to $\bdy C$ with dividing slope $s\neq 0$ in $(C,\xi'_{P_1,P_2})$, then $(P_1,P_2)$ is $2$-inconsistent. If $s<0$,  then there exists a convex torus between $\bdy C$ and $T$ with dividing slope $\infty$, which cuts off a copy $C'\subset C$ such that $(\xi'_{P_1,P_2})|_{C'}\in \Tight_0(C;\infty)$. By Lemma~\ref{lemma:s>0}, the same conclusion holds for $s>0$. It follows from the proof of Lemma~\ref{lemma:s=infty} that the restriction $(\xi'_{P_1,P_2})|_{C'}$ has the structure as described in the previous paragraph, and hence $(P_1,P_2)$ is $2$-inconsistent.
\end{proof}

\begin{lemma}
    For any contact structure in $\Tight_0(C;\infty)$, every convex torus parallel to $\bdy C$ in $C$ must have dividing slope $\infty$.
\end{lemma}

\begin{proof}
    Take any $\xi\in\Tight_0(C;\infty)$. Suppose that there exists a convex torus $T$ parallel to $\partial C$ in $(C,\xi)$ whose dividing slope $s$ differs from $\infty$. Then, inside the thickened torus cobounded by $\partial C$ and $T$, there is a convex torus $T'$ parallel to $\bdy C$ with dividing slope $n$, where $n$ is a negative integer. By Lemma~\ref{lemma:destabilization}, there exists a convex torus $T''$ parallel to $T'$ in $C$ with dividing slope $0$. Hence $T''$ separates off a copy of $C$, denoted by $C'$, with $\xi|_{C'}\in \Tight_0(C;0)$. 
    
    As in the proof of Lemma~\ref{lemma:tw=0}, $\xi|_{C'}$ can be described by some pair of paths $(P_1,P_2)$, that is, \( \xi|_{C'} = \xi_{P_1,P_2}|_{C'} = \xi'_{P_1,P_2}\). If $(P_1,P_2)$ is $2$-inconsistent, by Lemma~\ref{lemma:2-inconsistent}, there exists a convex torus parallel to $\bdy C'$ in $C'$ with dividing slope $\infty$. This would imply that $C$ has Giroux torsion, contradicting the assumption that $\xi\in\Tight_0(C;\infty)$.
    
    Now assume $(P_1,P_2)$ is $2$-consistent. In this case, note that $(C,\xi)$ is the union of $(C',\xi'_{P_1,P_2})$ and a basic slice in $\Tight^{min}(T^2\times [0,1];\infty,0)$. It then follows from Lemma~\ref{lemma:overtwisted} that $\xi$ is overtwisted, again contradicting the assumption that $\xi \in \Tight_0(C; \infty)$. 
\end{proof}

\begin{lemma}\label{lemma:overtwisted}
    For any pair of paths $(P_1, P_2)$ representing $q/p$ that is not totally $2$-inconsistent, denote by $\xi'_{P_1,P_2}\in \Tight_0(C;0)$ the restriction of $\xi_{P_1,P_2}$ on the complement of $L_{P_1,P_2}$. Then, gluing any basic slice in $\Tight^{min}(T^2\times [0,1];\infty,0)$ to $(C,\xi'_{P_1,P_2})$ will result in an overtwisted contact structure.
\end{lemma}

\begin{proof}
Based on the construction of the contact structure $\xi_{P_1,P_2}$ in Section~\ref{section:paths}, we have solid tori $V_1$ and $V_2$ with convex boundaries $T_1$ and $T_2$, each of dividing slope $q/p$. Inside these, we choose convex tori $T_1' \subset V_1$ and $T_2' \subset V_2$ with dividing slopes $(q/p)^a$ and $(q/p)^c$, respectively. In the model $C = V_1 \cup (S^1 \times P) \cup V_2$, the torus $T_1'$ has dividing slope $-1$, $T_2'$ has dividing slope $\infty$, and each $T_i$ (for $i=1,2,3$) has dividing slope $0$. Since $(P_1, P_2)$ is not totally $2$-inconsistent, we may arrange the basic slices $N_i$ cobounded by $T_i$ and $T_i'$ to have opposite signs for $i = 1, 2$. When attaching a basic slice $N_3$ from $\Tight^{\min}(T^2 \times [0,1]; \infty, 0)$ along $\partial C = T_3$, the sign of $N_3$ agrees with either $N_1$ or $N_2$. Denote $\partial N_3 = T_3 \cup T_3'$.


    If $N_1$ and $N_3$ have the same sign, then by \cite[Lemma 2.22]{emm}, there exists a convex annulus $A$ whose boundary consists of ruling curves of slope $0$ on $T_1'$ and $T_3'$, with exactly two dividing curves running from one boundary component to the other. Cutting $S^1 \times P$ along $A$ and rounding edges yields a convex torus $T$ parallel to $T_2'$ with dividing slope $\frac{1}{1+0-1}=\infty$. The thickened torus cobounded by $T$ and $T_2'$ contains convex Giroux torsion, since a torus $T_2$ with dividing slope $0$ lies inside it. Therefore, $V_2$ can be thickened to include a convex Giroux torsion layer, implying that the resulting contact structure is overtwisted.
    

     If $N_2$ and $N_3$ have the same sign, we consider a convex torus $T_2''$ with dividing slope $-1$ inside the thickened torus bounded by $T_2$ and $T_2'$. According to \cite[Lemma 2.22]{emm}, there exists a convex annulus $A$ whose boundary consists of ruling curves of slope $0$ on $T_2''$ and $T_3'$, with exactly two dividing curves running from one boundary component to the other. Extend $A$ along the $0-$sloped ruling curves from $T_2''$ to $T_2'$. Cutting $S^1 \times P$ along $A$ and rounding the edges yields a convex torus $T$ parallel to $T_1'$ with dividing slope $\frac{1}{0+0-1}=-1$. Thus, $V_1$ can be thickened to contain a convex Giroux torsion layer, which also implies that the resulting contact structure is overtwisted.
\end{proof}

\begin{lemma}\label{lemma:addtorsion}
For any pair of paths $(P_1, P_2)$ representing $q/p$ that is totally $2$-inconsistent, denote by $\xi'_{P_1,P_2}\in \Tight_0(C;0)$ the restriction of $\xi_{P_1,P_2}$ on the complement of $L_{P_1,P_2}$. Gluing one basic slice $(T^2\times [0,1]; \infty, 0)$ to $(C, \xi'_{P_1,P_2})$ will result in a tight contact structure $\xi$ on $C$ with dividing slope $\infty$ on $\partial C$, while gluing the other basic slice will result in an overtwisted
contact structure. Moreover, adding any amount of convex Giroux torsion to $(C, \xi)$ will result in a
tight contact structure.
\end{lemma}

\begin{proof}
    Use the same notation in the proof of Lemma~\ref{lemma:overtwisted}. Since the pair of paths $(P_1,P_2)$ is totally $2$-inconsistent, we can arrange the basic slices $N_1$ and $N_2$ to have the same sign, say, positive. By \cite[Lemma 2.22]{emm}, there exists a convex annulus $A$ with boundary $0$-sloped ruling curves on $T'_1$ and $T'_2$ whose dividing curves run from one boundary component to the other. Cutting along the annulus $A$ and rounding edges, one obtain a thickened torus $N$ with one boundary component $\partial C$ and the other an paralleled $\infty$-sloped convex torus $T$. Thus $N$ is a positive basic slice (the sign can be determined by the relative Euler class). Therefore, gluing a negative basic slice $B\in \Tight^{min}(T^2\times [0,1];\infty,0)$ would result in an overtwisted contact structure.

    Now we need to prove that gluing a positive basic slice $B'\in \Tight^{min}(T^2\times [0,1];\infty, 0)$ to $C$ results in a tight contact structure. Suppose we continue to glue sufficient amount of Giroux torsion layer. We will show that the resulting contact structure is tight, by a strategy analogous to the proof of \cite[lemma 6.15]{emm}. Assume that it is not true and then there exists an overtwisted disk $D$. Take a convex annulus $A'$ whose boundary consists of two $0$-sloped ruling curves on $T_3$ such that the complement of $A'$ are two solid torus $V'_1$ and $V'_2$. Since there exist $0$-twisting Legendrian curves on $A'$, we see that $V'_i$ ($i=1,2$) is tight for $V'_i\setminus V_i$ is a non-rotative outer layer. The annulus $A'$ can be isotoped into another annulus $A''$ such that $A''$ is disjoint with $D$. By Isotopy Discretization Lemma \cite{colin}, there is a sequence of convex annuli $A_0=A',\dots, A_n=A''$ such that $A_{i+1}$ is obtained from $A'_i$ by a single bypass attachment. As shown in the proof of \cite[Lemma 6.16]{emm} and the paragraph before Lemma 6.16, one can see that any $A_i$ is contained in an $I$-invariant neighborhood that is contact isotopic to an $I$-invariant neighborhood of $A'$. Therefore, each complement of $A''$ must be tight, which contradicts to the assumption that $A''$ is disjoint with $D$.
\end{proof}

\begin{lemma} \label{lemma:tor>0}
For any $k\in\mathbb{Z}$ and $t\in \frac{1}{2}\mathbb{N}$,   $|\Tight_{t}(C;k)|=2l(p,q).$  This number coincides with  the number of totally $2$-inconsistent pairs of paths representing $q/p$.
\end{lemma}

\begin{proof}
This is simlilar to the proof of \cite[Lemma 6.19]{emm}, except the proof of the following claim.

Claim: Suppose that $\xi$ and $\xi'$ are two contact structures on $\Tight_{0}(C; 0)$ associated to totally 2-inconsistent pairs of decorated paths $(P_1, P_2)$ and $(P'_1, P'_2)$ representing $q/p$ such that $(P_1, P_2)\neq (P'_1, P'_2)$. Adding convex $t$ Giroux torsion to $\xi$ and $\xi'$ results in distinct contact structures $\xi_{t}$ and $\xi'_{t}$, respectively. 


Suppose $t \in \mathbb{N}$. If the resulting contact structures $\xi_{t}$ and $\xi'_{t}$ represent the same element in $\Tight_{t}(C;k)$, then they must be homotopic as 2-plane fields, which in turn implies that $\xi$ and $\xi'$ are also homotopic. In general, the operation of adding $t$ convex Giroux torsion is injective for plane fields, unlike for contact structures, where it is not injective. Now, consider performing contact $(-\frac{1}{2})$-surgery along $L_{P_1,P_2}$ and $L_{P'_1,P'_2}$. This yields two non-isotopic tight contact structures on the same lens space. Since $\xi$ and $\xi'$ are homotopic, these two resulted tight contact lens spaces are homotopic. However, according to \cite[Proposition 4.24]{h}, any two non-isotopic tight contact structures on a lens space are not homotopic. This leads to a contradiction.

Suppose $t \in \frac{1}{2}\mathbb{N}\setminus\mathbb{N}$. If the resulting contact structures $\xi_{t}$ and $\xi'_{t}$ represent the same element in $\Tight_{t}(C;k)$, then $\xi_{t+\frac{1}{2}}$ and $\xi'_{t+\frac{1}{2}}$ represent the same element in $\Tight_{t+\frac{1}{2}}(C;k)$. This is impossible by the previous argument.
\end{proof}

\section{Classification of non-loose Legendrian torus knots in \texorpdfstring{$S^1\times S^2$}{S1 x S2}}\label{section:proofs}




\begin{proof}[Proof of Theorem~\ref{thm:tw>0}]
Let $i$ be a positive integer. By Lemma~\ref{lemma:s=n>0}, there are at most $2n(p,q)$ non-loose Legendrian $(p,q)$-torus knots with $\tw=i$ and $\tor=0$. 

We show that all $L_{\pm, k}^{i}$ can be realized as Legendrian knots in Figure~\ref{Figure:tw>0}. First, a Kirby calculus argument -changing the coefficient $(-\frac{p}{p'})$ in \cite[Figure 26]{emm} to $(-\frac{q}{q'})$- shows that $L_{\pm}$ is a smooth torus knot $T_{p,q}$ in $S^1 \times S^2$. Consider the case $\tw = i = 1$. Performing Legendrian surgery along $L_{\pm}$ in the first row of Figure~\ref{Figure:tw>0} yields a tight $L(q,p) \# L(-q,p)$, so $\tor(L_{\pm}) = 0$. By Lemma~\ref{lemma:surgerychar}, the smooth surgery coefficient with respect to $\lambda$ must be $0$, which equals $\tw - 1$. Hence, $\tw = 1$. Now suppose $\tw = i > 1$. Performing Legendrian surgery along $L_{\pm}$ in the second row of Figure~\ref{Figure:tw>0} yields a tight contact 3-manifold $M(S^2; \frac{q'}{q}, -\frac{q'}{q}, -\frac{1}{i-1})$, so $\tor(L_{\pm}) = 0$. By Lemma~\ref{lemma:surgerychar}, the smooth surgery coefficient with respect to $\lambda$ must be $i - 1$, which equals $\tw - 1$. Therefore, $\tw = i$.

Each diagram in Figure~\ref{Figure:tw>0} corresponds to $n(p,q)$ distinct non-loose Legendrian $(p,q)$ torus knots. Indeed, after performing Legendrian surgery and coverting the negative rational contact surgeries to  sequences of contact $-1$ surgeries, we obtain $n(p,q)=|\Tight(L(q,q')|\cdot|\Tight(L(q, -q')|$ contact surgery diagrams in which every contact surgery coefficient is $-1$; all of these contact structures are Stein fillable. So according to \cite{lm},  the rotation numbers distinguish $n(p,q)$ pairwise non-isotopic tight contact structures  on either  $L(q,p) \# L(-q,p)$ or $M(S^2; \frac{q'}{q}, -\frac{q'}{q}, -\frac{1}{i-1})$. As in the proof of Lemma~\ref{lem:tw=0distinct}, each diagram corresponds to $n(p,q)$ non-loose Legendrian $(p,q)$ torus knots which are pairwise coarse inequivalent.



Recall that $n(p,q) = |\Tight(V_1; (q/p)^a)| \cdot |\Tight(V_2; (q/p)^c)|$. According to the proof of Lemma~\ref{lemma:2-inconsistent}, every element in $\Tight(V_1; (q/p)^a) \times \Tight(V_2; (q/p)^c)$ can be extended to two elements in $\Tight(C; 0)$ corresponding to two distinct $2$-inconsistent pairs of paths.
Now, consider any $2$-inconsistent pair of paths $(P_1, P_2)$ representing $q/p$, and let $C$ be the complement of $L_{P_1,P_2}$. By Lemma~\ref{lemma:2-inconsistent}, there exists a copy $C' \subset C$ such that $(\xi_{P_1,P_2})|_{C'} \in \Tight_0(C; \infty)$ and $(\xi_{P_1,P_2})|_{C \setminus C'} \in \Tight^{\text{min}}(T^2 \times [0,1]; 0, \infty)$. This implies that $C \setminus C' = T^2 \times [0,1]$ is a basic slice, which we may assume to be negative. 
For any integer $i > 0$, we factor $T^2 \times [0,1]$ into negative basic slices whose boundaries correspond to slopes $0, 1, \dots, i, \infty$ in the Farey graph. Then, there exists a convex torus in $T^2 \times [0,1]$ with two dividing curves of slope $i$, which cuts off the complement of a Legendrian $(p,q)$-torus knot $L^i_-$ with $\tw(L^i_-) = i$, as desired. Furthermore, attaching a negative basic slice to this complement yields the complement of $L^{i-1}_-$, i.e., $S_{-}(L^i_-) = L^{i-1}_-$. On the other hand, attaching a positive basic slice (corresponding to a positive stabilization) results in an overtwisted contact structure, so $S_+(L^i_-)$ is loose. By iteratively attaching negative basic slices, the boundary slope eventually becomes $0$, meaning $S_-^{i}(L^i_-) = L_{P_1,P_2}$. Furthermore, since we have assumed that $T^2\times [0,1]$ is a negative basic slice, $S_+(L_{P_1,P_2})$ is loose, and consequently, it follows from Proposition~\ref{prop:wings} that $S_{-}^{j}(L_{-}^{1})$ is non-loose for any $j>0$. Analogously, if we instead assume the basic slice $T^2\times [0,1]$ is positive, then we obtain  $L_+^i$ with corresponding stabilization properties. 
Moreover, $L_{+, k}^i$ and $L_{-, k}^i$ are not coarsely equivalent.

Additionally, we can apply Corollary~\ref{coro:signoftw>0} and Remark~\ref{rmk:sign} to determine which diagram in Figure~\ref{Figure:tw>0} corresponds to \( L_+ \) or \( L_- \).
\end{proof}


From the proof of Theorem~\ref{thm:tw>0}, we have
\begin{corollary} \label{coro:tw>0stabilize}
    Any non-loose Legendrian $(p,q)$-torus knot with $\tw>0$ and $\tor=0$ can be stabilized to one with $\tw=0$ and $\tor=0$, which can be described as $L_{P_1,P_2}$ for some $2$-inconsistent pair of paths $(P_1,P_2)$ representing $q/p$.
\end{corollary}

\begin{proof}[Proof of Theorem~\ref{thm:tw=0}]
By Lemma~\ref{lemma:tw=0}, there are exactly $m(p,q)$ coarse equivalences of Legendrian $(p,q)$-torus knots in contact $S^1\times S^2$ with tight complements. Since there are exactly $2$ coarse equivalences of Legendrian $(p,q)$-torus knots with $\tw=0$ in the unique tight contact $S^1\times S^2$,  there are exactly $m(p,q)-2$ coarse equivalences of non-loose Legendrian $(p,q)$-torus knots with $\tw=0$  in contact $S^1\times S^2$.  By the proof of Lemma~\ref{lemma:tw=0},  these knots come from the surgery diagram in Figure~\ref{Figure:tw=0}.
\end{proof}

Here we present the stabilization properties of non-loose Legendrian torus knots with $\tw=0$. Consider any decorated pair of paths $(P_1,P_2)$ representing $q/p$. Suppose that the total number of continued fraction blocks in $P_1$ and $P_2$ is $N$. We may denote $P_1=(A_2,A_4,\ldots,A_{2n})$ and $P_2=(B_1,B_3, \ldots, B_{2m-1})$, where $B_1$ has length $1$ and $m+n=N$. Let $s_j$ be the slope in $A_j$ or $B_j$ which is farthest from $q/p$, and denote $n_j= |s_j\bigcdot \frac{q}{p}|$ for $1\leq j\leq N$. Let $m_j$ be half the number of $j$-inconsistent pairs of paths representing $q/p$ for $2\leq j\leq N$. According to the proof of Lemma~\ref{lemma:s=infty}, $m_2=n(p,q)$.


Assume that $(P_1,P_2)$ is $k$-inconsistent for some $2\leq k \leq N$ and not compatible with any $(k+1)$-inconsistent pair of paths.
Observe that the total number of such pairs is $2(m_k-m_{k+1})$. Note that $(P_1,P_2)$ can be converted into the unique compatible $k$-inconsistent pair of paths $(P_1^r,P_2^r)$ for $2\leq r\leq k$, where $(P^k_1,P^k_2)=(P_1,P_2)$. 


\begin{proposition}\label{prop:wings}
Assume that the ambient contact structure is not $\xi_{std}$. Given an $k$-inconsistent pair of paths $(P_1,P_2)$ representing $q/p$, assume that $k$ is even and all the the basic slices in the continued fractions blocks $A_2,\ldots, A_{k-2}, B_1, \ldots, B_{k-1}$ are positive while some in $A_k$ are negative. Then we have 
$$S_-^{n_{k-1}}(L_{P_1,P_2}) ~\text{is loose and}~ S_+^lS_-^i(L_{P_1,P_2}) ~\text{is non-loose for any}~ i < n_{k-1} ~\text{and}~ l \geq 0;$$
$$S_+^{n_{k-1}}(L_{-P_1,-P_2}) ~\text{is loose and}~ S_+^iS_-^l(L_{-P_1,-P_2}) ~\text{is non-loose for any}~ i < n_{k-1} ~\text{and} ~l \geq 0.$$

Moreover, $S_+^{n_{r-1}-n_{r-2}}(L_{P_1^r,P_2^r})$ and $S_-^{n_{r-1}-n_{r-2}}(L_{P_1^{r-1},P_2^{r-1}})$ are coarsely equivalent for $2<r\leq k$.

When $i$ is odd, the same result holds if all basic slices in the continued fraction blocks $A_2, \ldots, A_{i-1}$, $B_1, \ldots, B_{i-2}$ are negative and some in $B_i$ are positive.
\end{proposition}

\begin{proof}
    This is similar to the proof of \cite[Proposition 7.5 and 7.10]{emm}.
\end{proof}

Denote by $W(L_{P_1,P_2})$ (respectively $W(L_{-P_1,-P_2})$) the set $\{S_+^lS_-^i(L_{P_1,P_2}) \text{ for } i < n_{k-1} \text{ and } l \geq 0\}$ (respectively $\{S_+^iS_-^l(L_{-P_1,-P_2}) \text{ for } i < n_{k-1} \text{ and } l \geq 0\}$).  The value $n_{k-1}$ is then defined as the \dfn{depth} of $W(L_{P_1,P_2})$ (respectively $W(L_{-P_1,-P_2})$) along the line of slope $1$ (respectively $-1$). 

Assume that $(P_1, P_2)$ is $k$-inconsistent for some $2 \leq k \leq N$ and is not compatible with any $(k+1)$-inconsistent pair of paths. Let $(P_1, P_2) = (P_1^k, P_2^k)$. The set $\bigcup_{r=2}^k W(L_{P_1^r, P_2^r})$ (respectively, $\bigcup_{r=2}^k W(L_{-P_1^r, -P_2^r})$) is called a \dfn{wing} of $L_{P_1, P_2}$ (respectively, $L_{-P_1, -P_2}$). The value $n_{k-1}$ is then referred to as the \dfn{depth} of the wing of $L_{P_1, P_2}$ (respectively, $L_{-P_1, -P_2}$). The set $$\bigcup_{r=2}^k W(L_{P_1^r,P_2^r})\bigsqcup \bigcup_{r=2}^k W(L_{-P_1^r,-P_2^r})$$ is called a \dfn{pair of wings}. Note that $\bigcup_{r=2}^k W(L_{P_1^r,P_2^r})$ and $\bigcup_{r=2}^k W(L_{-P_1^r,-P_2^r})$ are symmetric with respect to a vertical line.

\bigskip
\begin{figure}[htb]
\begin{overpic}
{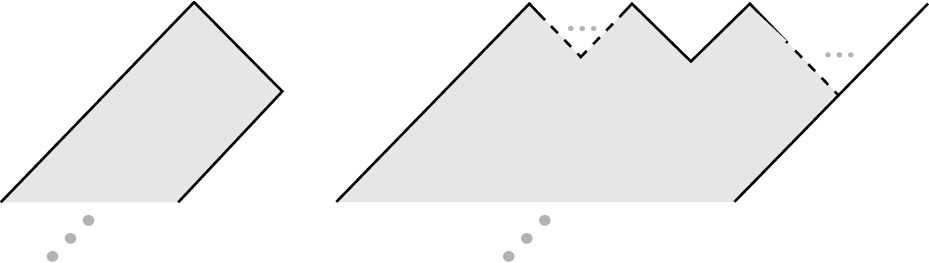} 
\put(75, 130){$L_{-P_1,-P_2}$}
\put(277, 130){$L_{-P_1^r,-P_2^r}$}
\put(335, 130){$L_{-P_1^{r-1},-P_2^{r-1}}$}
\put(230, 130){$L_{-P_1,-P_2}$}
\put(410, 130){$L_{-P_1^2,-P_2^2}$}
\end{overpic}
\caption{The left subfigure is $W(L_{-P_1,-P_2})$. The right subfigure is the wing $\bigcup_{r=2}^k W(L_{-P_1^r,-P_2^r})$. The distance between the two peaks corresponding to \(L_{-P_1^r,-P_2^r}\) and \(L_{-P_1^{r-1},-P_2^{r-1}}\) is \(2(n_{r-1}-n_{r-2})\).}
\label{Figure:wing}
\end{figure}
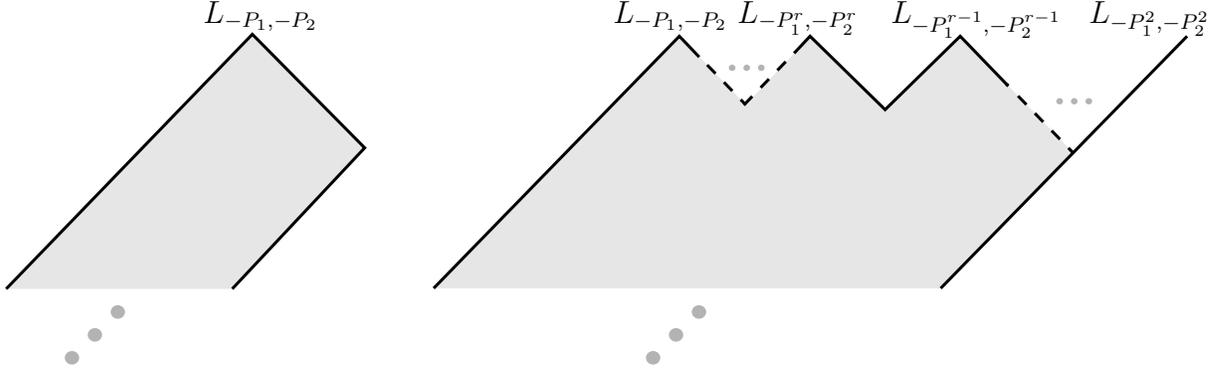

By Proposition~\ref{prop:wings}, we have 
\begin{corollary}\label{coro:numberofwings}
There are exactly $m_2$ pairs of wings for non-loose Legendrian $(p,q)$-torus knots with $\tor=0$. Precisely, 
For $2\leq j\leq N$, there are exactly $m_j-m_{j+1}$ pairs of wings each of which have depth $n_{j-1}$, where $m_{N+1}=0$.
\end{corollary}

\begin{proof}[Proof of Theorem~\ref{thm:Legmountain}]
By Lemma~\ref{lemma:tw=0}, any such knots with $\tw=0$ can be realized as $L_{P_1,P_2}$, corresponding to the peaks in Figure~\ref{Figure:mountain}. The case $\tw>0$ follows from Theorem~\ref{thm:tw>0} and Corollary~\ref{coro:tw>0stabilize}, corresponding to the portion of two lines of slopes $1$ and $-1$ that lie above the peaks in Figure~\ref{Figure:mountain}. The case $\tw<0$ is covered by Theorem~\ref{lemma:destabilization} and Proposition~\ref{prop:wings}, corresponding to the pair of wings in Figure~\ref{Figure:mountain}. The number of such pairs of components is given by $n(p,q)$, as stated in Corollary~\ref{coro:numberofwings}.   
\end{proof}

\begin{proof}[Proof of Theorem~\ref{thm:tor>0}]
It follows from Lemma~\ref{lemma:addtorsion} and Lemma~\ref{lemma:tor>0}.
\end{proof}

\begin{proof}[Proof of Corollary~\ref{coro:destabilization}]
It follows from Theorem~\ref{thm:Legmountain} and \ref{thm:tor>0}. A non-loose Legendrian torus knot with $\tw=0$ destabilizes if and only if its corresponding pair of paths is $2$-inconsistent, according to Lemma~\ref{lemma:2-inconsistent}.
\end{proof}

\section{Classification of non-loose transverse torus knots in \texorpdfstring{$S^1\times S^2$}{S1 x S2}}
We need the following lemma which implies that classifying transverse knots is equivalent to classifying Legendrian knots up to negative stabilization.

\begin{lemma}\label{lemma:stablesimple}\cite{eh1}
    Let $K_1$ and $K_2$ be two transverse knots with a given knot type $\mathcal{K}$. The knots $K_1$ and $K_2$ are transversely isotopic if and only if there exist two Legendrian approximations $L_1$ and $L_2$ of $K_1$ and $K_2$ respectively, such that for some integers $r_1$ and $r_2$, two negative stabilization $S^{r_1}_-(L_1)$ and $S^{r_2}_-(L_2)$ are Legendrian isotopic.
\end{lemma}

\begin{remark}
    The above lemma was originally proved in \cite[Theorem 2.10]{eh1} for transverse isotopy but the proof still works for coarse equivalence.
\end{remark}



\begin{corollary} \label{coro:transversetor=0}
    Suppose $-q>p>0$. Let $t(p,q)=\sum_{k=2}^{N}(m_k-m_{k+1})n_{k-1}$. Then the number of non-loose transverse $(p,q)$-torus knots in $S^1\times S^2$ with $\tor=0$ is exactly $t(p,q)$. These knots are denoted by 
    \[T_{i,j,k} ~\text{ for }~ 1\leq i \leq n_{k-1},~ 1\leq j \leq m_k-m_{k+1}, ~\text{ and } 1\leq k \leq N.\]
    Moreover, 
    \[S(T_{i,j,k})=T_{i-1,j,k} ~\text{ for }~ 1<i\leq n_{k-1} ~\text{ and } S(T_{1,j,k}) ~\text{ is loose}.\]
    In particular, given $j$ and $k$, these knots can be described by an interval as Figure~\ref{figure:transverse}. There are exactly $n(p,q)$ such intervals.
\end{corollary}
\begin{proof}
    By Lemma~\ref{lemma:stablesimple}, it suffices to consider those non-loose Legendrian torus knots with $\tw\leq 0$ and $\tor=0$ that remain non-loose after arbitrary many negative stabilizations. By Lemma~\ref{lemma:destabilization} and \ref{lemma:tw=0}, every such knot is coarsely equivalent to either $L_{-P_1,-P_2}$ for some decorated pair of paths $(P_1,P_2)$ representing $q/p$, or a stabilization of $L_{-P_1,-P_2}$. (Here $(P_1,P_2)$ is described as Proposition~\ref{prop:wings}.)
    
    For $2\leq k\leq N$, let $(P_1,P_2)$ be a $k$-inconsistent pair of paths that is not compatible with any $(k+1)$-inconsistent pair. By Corollary~\ref{coro:numberofwings}, the wing of $L_{-P_1,-P_2}$ has depth $n_{k-1}$, and the total number of such wings is $m_k-m_{k+1}$. Therefore, by Proposition~\ref{prop:wings}, the non-loose transverse torus knots in $(S^1\times S^2, \xi_{-P_1,-P_2})$ are exactly the transverse push-offs of $S^{i}_+(L_{-P_1,-P_2})$ for $0\leq i\leq n_{k-1}-1$. An interval for non-loose representatives of transverse torus knots with $\tor=0$ is shown in Figure~\ref{figure:transverse}.
\end{proof}

\begin{figure}[htb]
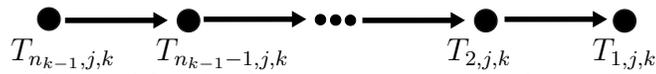

\begin{overpic} 
{transverse.eps}
\put(210,-10){$T_{1,j,k}$}
\put(155,-10){$T_{2,j,k}$}
\put(45,-10){$T_{n_{k-1}-1,j,k}$}
\put(-10,-10){$T_{n_{k-1},j,k}$}
\end{overpic}
\vspace{1mm}
\caption{An interval for non-loose transverse torus knots. For the case $\tor=0$, each dot represents a unique non-loose transverse representative. For the case $\tor>0$, the interval reduces to a single dot, which represents an infinite family of transverse representatives distinguished by convex Giroux torsion.}
\label{figure:transverse}
\end{figure}

\begin{proof}[Proof of Corollary~\ref{coro:transversetor>0}]
    This follows from Lemma~\ref{lemma:stablesimple} and Theorem ~\ref{thm:tor>0}. Note that for any totally $2$-inconsistent pair of paths, there exists no $k$-inconsistent pair of paths for $k>2$ that is compatible with it. Then each interval consists of a single dot in Figure~\ref{figure:transverse}.
\end{proof}



\section{Euler classes of contact structures supporting non-loose torus knots}\label{section:Euler}

Suppose $-q>p>0$. Let $(P_1,P_2)$ be a decorated pair of paths representing $q/p$. Suppose the vertices of \(P_1\) are \(p_1, \dots, p_k\), where \(p_1 = q/p\) and \(p_k = \infty = -1/0\), and those of \(P_2\) are \(q_1, \dots, q_l\), where \(q_1 = q/p\) and \(q_l = -1\).
From \cite[Proof of Proposition 3.22]{h}, we can compute the Euler class of the contact structure $\xi_{P_1,P_2}$ on $S^1\times S^2$ by
\begin{align}
e(\xi_{P_1,P_2})=\sum_{i=1}^{k-1}\epsilon_i\left((p_{i+1}\ominus p_i)\bigcdot \frac{0}{1}\right)~+~\sum_{i=1}^{l-1}\epsilon'_i\left((q_i\ominus q_{i+1})\bigcdot \frac{0}{1}\right),\label{formula:e}
\end{align}
where $\epsilon_i$ is the sign of the edge from $p_{i+1}$ to $p_i$ and $\epsilon'_i$ is the sign of the edge from $q_i$ to $q_{i+1}$.

We may denote $P_1=(A_1,A_3,\ldots,A_{2n-1})$ and $P_2=(B_2,B_4, \ldots, B_{2m})$. Let $s_k$ be the slope in $A_k$ or $B_k$ which is farthest from $q/p$, and denote $n_k= |s_k\bigcdot \frac qp|$.

\begin{lemma}\label{lem:Euler} Let $e(\xi_{P_1, P_2})$ be the Euler class of the contact structure $\xi_{P_1, P_2}$ on $S^1\times S^2$. Then
\begin{enumerate}
\item $e(\xi_{P_1, P_2})+e(\xi_{-P_1, -P_2})=0$, 
\item $e(\xi_{P_1, P_2})=0$  if and only if the signs of the basic slices in both $P_1$ and $P_2$ are all positive or all negative, and
\item $e(\xi_{P_1, P_2})=0$  if and only if $\xi_{P_1, P_2}$ is the standard tight contact structure on $S^1\times S^2$.
\end{enumerate}
\end{lemma}

\begin{proof}
(1) This follows immediately from the formula~\ref{formula:e}, since it only changes the terms $\epsilon_i$ and $\epsilon'_i$ to $-\epsilon_i$ and $-\epsilon'_i$ from $e(\xi_{P_1,P_2})$ to $e(\xi_{-P_1,-P_2})$.

(2) Let $(P_1',P_2')$ be a decorated pair of paths with all basic slices positively signed. Then $$e(\xi_{P'_1,P'_2})=(\frac{-1}{0}\ominus\frac{q}{p})\bigcdot \frac{0}{1}+(\frac{q}{p}\ominus\frac{-1}{1})\bigcdot \frac{0}{1}=(-1-q)+(q+1)=0.$$ By (1), $e(\xi_{-P'_1,-P'_2})=e(\xi_{P'_1,P'_2})=0$. Now let $(P_1,P_2)$ be any decorated pair of paths in which the signs of basic slices are not all the same. To compute the Euler class of  $\xi_{P_1,P_2}$, it suffices to consider only the negative basic slices in $(P_1,P_2)$, since $e(\xi_{P_1,P_2})=e(\xi_{P_1,P_2})-e(\xi_{P'_1,P'_2})$. In the computation, because $$(p_{i+1}\ominus p_i)\bigcdot \frac{0}{1}>0 ~\text{for}~ i=1,\ldots k-1,$$ and $$(q_i\ominus q_{i+1})\bigcdot \frac{0}{1}<0~ \text{for}~ i=1,\ldots, l-1,$$ negative basic slices in $P_1$ contribute negative terms, while those in $P_2$ contribute positive terms. Consider the sequence of continued fraction blocks $(A_1,B_2,A_3,\ldots)$.

First, assume that the first negative basic slice in the sequence appears in $B_j$ for some $j\in\{2,4,\ldots,2m\}$. Denote the two vertices of this basic slice by $q_t$ and $q_{t+1}$, and the first vertex in $A_{j+1}$ by $p_r=s_{j-1}$. By assumption, all basic slices in $A_1,A_3,\ldots,A_{j-1}$ are positive and can therefore be omitted in the computation. Note that there is an edge between $s_{j-1}$ and every vertex in $B_j$, and hence $q_{t}\ominus q_{t+1}=s_{j-1}$. Then we have
\begin{align*}
    e(\xi_{P_1,P_2})\geq &\sum_{i\geq r}^{k-1}(-2)\left((p_{i+1}\ominus p_i)\bigcdot \frac{0}{1}\right)+(-2)\left((q_{t}\ominus q_{t+1})\bigcdot \frac{0}{1}\right)\\
    =&-2\left((\frac{-1}{0}\ominus s_{j-1})\bigcdot \frac{0}{1}\right)-2(s_{j-1}\bigcdot\frac{0}{1})\\
    =&-2\left(((\frac{-1}{0}\ominus s_{j-1})\oplus s_{j-1})\bigcdot \frac{0}{1}\right)=-2(\frac{-1}{0}\bigcdot\frac{0}{1})=2>0.
\end{align*}

Next, assume that the first negative basic slice appears in $A_j$ for some $j\in\{1,3,\ldots,2n-1\}$. If $j>1$, by the same reasoning as above, we have 
\begin{align*}
    e(\xi_{P_1,P_2})\leq &~ 2(s_{j-1}\bigcdot \frac{0}{1})+(-2)\left((s_{j-1}\ominus\frac{-1}{1})\bigcdot \frac{0}{1}\right)\\
    =&-2\left(((s_{j-1}\ominus\frac{-1}{1})\ominus s_{j-1})\bigcdot \frac{0}{1}\right)=-2(\frac{1}{-1}\bigcdot\frac{0}{1})=-2<0.
\end{align*}
Suppose $j=1$. Note that $A_1$ consists of a single negative basic slice, as it has length $1$. We then consider $(-P_1,-P_2)$, whose first negative basic slice must appear in some continued fraction block within $(B_2, A_3, B_4, \dots)$. From the preceding results, we know that $e(\xi_{-P_1,-P_2})\neq0$, and hence $e(\xi_{P_1,P_2})\neq 0$ by (1).

(3) For any pair of paths $(P_1,P_2)$, there is a Legendrian knot $L_{P_1,P_2}$ in $(S^1\times S^2, \xi_{P_1,P_2})$. By Lemma~\ref{lemma:tw=0}, any Legendrian $(p,q)$-torus knot with $\tw=0$ and tight complement in $S^1\times S^2$ can be constructed as such $L_{P_1,P_2}$. By Remark~\ref{rmk:tw=0intight}, there are exactly $2$ Legendrian $(p,q)$-torus knots with $\tw=0$ in $(S^1\times S^2, \xi_{std})$. Since $e(\xi_{std})=0$, these two knots must correspond to $L_{\pm P'_1,\pm P'_2}$ with the signs of basic slices in $(P'_1,P'_2)$ all positive, according to (2). Hence $\xi_{\pm P'_1,\pm P'_2}$ is the standard tight contact structure on $S^1\times S^2$.
\end{proof}

\begin{lemma}\label{lem:numberEuler}
    For all decorated pairs of paths representing $q/p$, we have $e(\xi_{P_1,P_2})=\pm2k$ for all $k\in\{0,1,2,\ldots,-q-1\}$.
\end{lemma}

\begin{proof}
We may label the continued fraction blocks as $P_1=(A_1,A_3,\ldots,A_{2n-1})$ and $P_2=(B_2,B_4,\ldots, B_{2n})$,  noting that several other cases can be handled analogously. To determine all possible values of $e(\xi_{P_1,P_2})$, we apply formula~\ref{formula:e}, starting from the decoration where the signs of basic slices in $(P_1,P_2)$ are all positive, i.e. $e(\xi_{P_1,P_2})=0$. We then proceed along the sequence $(B_{2n}, A_{2n-1},\ldots, B_2, A_1)$, sequentially analyzing the contributions to the Euler class resulting from changing the signs of basic slices from positive to negative in each continued fraction block. Denote 
\[\Delta_j=
    \begin{cases}
    2|s_{j-1}\bigcdot\frac{0}{1}|, & \text{if }j\geq2,\\
    2|(s_1\ominus \frac{q}{p})\bigcdot\frac{0}{1}|, & \text{if }j=1,
    \end{cases}
\]
and
\[e_j=
    \begin{cases}
    2|(s_j\ominus s_{j-2})\bigcdot \frac{0}{1}|, &\text{if } j\geq3, \\
    2|(s_j\ominus \frac{q}{p})\bigcdot \frac{0}{1}|, &\text{if } j=1,2.
    \end{cases}
\]
In particular, we recall that $s_{2n}=-1/1$, $s_{2n-1}=-1/0$, and $s_1=(q/p)^a=q''/p''.$

In each block $B_j$, changing the sign of any basic slice from positive to negative contributes a positive term $\Delta_j$ when computing the Euler class. This is because any two adjacent vertices $q_t$ and $q_{t+1}$ in $B_j$ satisfy $q_{t}\ominus q_{t+1}=s_{j-1}$. Similarly, in each $A_j$, changing the sign of any basic slice from positive to negative contributes a negative term $-\Delta_j$. As a result, the set of all possible values generated by negative basic slices in  $B_j$ ranges from $0$ to $e_j$, with step size $\Delta_j$, while for $A_j$, the range is from $-e_j$ to $0$, also with step size $\Delta_j$. Here, the value $0$ is achieved by keeping all basic slices positive in every continued fraction block, whereas the value $e_j$ (or $-e_j$) corresponds to changing all basic slices in the continued fraction block $B_j$ (or $A_j$) from positive to negative, while keeping all others  positive.

We start with the case in which all basic slices are positive to compute the Euler class. The initial continued fraction block $B_{2n}$ can generate terms ranging from $0$ to $e_{2n}$ with step size $$\Delta_{2n}=2|s_{2n-1}\bigcdot\frac{0}{1}|=2|\frac{-1}{0}\bigcdot\frac{0}{1}|=2.$$  We then proceed to $A_{2n-1}$, which can generate terms ranging from $-e_{2n-1}$ to $0$, with step size $\Delta_{2n-1}$. 
Note that $\Delta_{2n-1}=e_{2n}+2$. For each term $-r\Delta_{2n-1}$ generated by basic slices in $A_{2j-1}$, combining with all terms generated by basic slices in  $B_{2n}$ yields the value set $$\{-r\Delta_{2n-1},-r\Delta_{2n-1}+2,\ldots,-r\Delta_{2n-1}+e_{2n}=-(r-1)\Delta_{2n-1}-2\}.$$ That is, the terms generated by all basic slices in $B_{2n}$  exactly fill, with step size $2$, the interval between two consecutive terms generated by basic slices in $A_{2j-1}$. Therefore,  the Euler class can take all values from $-e_{2n-1}$ to $e_{2n}$ with step size $2$. 

By induction on the index of the continued fraction blocks in the sequence, assume that for \( j \geq 2 \), the Euler class generated by negative basic slices in \( (B_{2n}, A_{2n-1}, \dots, B_{2j}) \) takes all values from 
$-\sum_{i=j+1}^{n}e_{2i-1}$ to $\sum_{i=j}^{n}e_{2i}$
with step size \( 2 \).
Now consider the computation proceeding to $A_{2j-1}$, which can generate terms from $-e_{2j-1}$ to $0$ with step size $\Delta_{2j-1}$. Note that $$\Delta_{2j-1}=2|s_{2j-2}\bigcdot\frac{0}{1}|=2|(s_{2j-2}\ominus \frac{-1}{1})\bigcdot\frac{0}{1}-1|=2|\sum_{i=j}^{n}(s_{2i-2}\ominus s_{2i})\bigcdot\frac{0}{1}|+2=\sum_{i=j}^{n}e_{2i}+2.$$
The positive terms $\{2, 4, \ldots, \sum_{i=j}^{n}e_{2i}\}$ generated by $(B_{2n},A_{2n-1},\ldots,B_{2j})$ therefore exactly fill, with step size $2$, each interval $[-r\Delta_{2j-1}, -(r-1)\Delta_{2j-1}]$ between two consecutive terms generated by basic slices in $A_{2j-1}$. Consequently, the Euler class take all values from $-e_{2j-1}$ to $0$ with step size $2$. Combining this with the term  $-\sum_{i=j+1}^{n}e_{2i-1}$ generated by $(B_{2n},A_{2n-1},\ldots,\\B_{2j})$, the Euler class takes all values from $-\sum_{i=j}^{n}e_{2i-1}$ to $-\sum_{i=j+1}^{n}e_{2i-1}$ with step size $2$. Together with the earlier range, it follows that the Euler class attains all values from $-\sum_{i=j}^{n}e_{2i-1}$ to $\sum_{i=j}^{n}e_{2i}$ with step size $2$. Note that $$\Delta_{2j-2}=2|s_{2j-3}\bigcdot\frac{0}{1}|=2|(s_{2j-3}\ominus \frac{-1}{0})\bigcdot\frac{0}{1}-1|=2|\sum_{i=j}^{n}(s_{2i-3}\ominus s_{2i-1})\bigcdot\frac{0}{1}|+2=\sum_{i=j}^{n}e_{2i-1}+2.$$
So a similar argument applies when passing from  $(B_{2n},A_{2n-1},\ldots,A_{2j-1})$ to $B_{2j-2}$ for $j\geq 2$. Therefore, when the computation reaches $B_2$, the Euler classes can take all values from $-\sum_{i=2}^{n} e_{2i-1}$ to $\sum_{i=1}^{n}e_{2i}$ with step size $2$. 

Finally, we proceed to $A_1$, whose basic slice can  generate only the values $-e_1$ or $0$. Then the Euler class then takes values either from $-\sum_{i=1}^n e_{2i-1}$ to $-e_1+\sum_{i=1}^ne_{2i}$, or from $-\sum_{i=2}^n e_{2i-1}$ to $\sum_{i=1}^ne_{2i}$, with step size $2$. In fact, we have
\[-\sum_{i=1}^n e_{2i-1}=-2|(\frac{-1}{0}\ominus\frac{q}{p})\bigcdot\frac{0}{1}|=2q+2, \sum_{i=1}^ne_{2i}=2|(\frac{-1}{1}\ominus\frac{q}{p})\bigcdot\frac{0}{1}|=-2q-2,\]
\[-e_1+\sum_{i=1}^ne_{2i}=2q'+(-2q-2)=-2(q''+1)>0, \text{ and } -\sum_{i=2}^n e_{2i-1}<0.\]
Therefore, the Euler class takes all values from $2q+2$ to $-2q-2$ in steps of 2, thereby establishing that this process exhausts all possible distributions of negative basic slices and, consequently, all possible values of the Euler class.
\end{proof}

\begin{lemma}\label{lem:numberEuler2}
    For all totally $2$-inconsistent decorated pairs of paths representing $q/p$, we have $e(\xi_{P_1,P_2})=\pm2k$ for $k\in\{k_e+q+1,k_e+q+2,\ldots,-q-1\}$, where \[k_e=|(s_1\ominus \frac{q}{p})\bigcdot \frac{0}{1}|+|(s_2\ominus \frac{q}{p})\bigcdot \frac{0}{1}|=
    \begin{cases}
    -(q'+|a_m+1|q''), &\text{if }a_m<-2,\\
    -(|b_n+1|q'+q''), &\text{if }a_m=-2.
    \end{cases}\]
\end{lemma}
\begin{proof}
Following the same  strategy as in the proof of Lemma~\ref{lem:numberEuler}, until we proceed to $A_3$, the Euler class takes all values from $-\sum_{i=2}^{n}e_{2i-1}$ to $\sum_{i=2}^{n}e_{2i}$ with step size $2$. We now consider the two remaining blocks $B_2$ and $A_1$.
There are exactly two decorations on $B_2$ and $A_1$ that are totally $2$-inconsistent: all basic slices in $B_2$ are positive and all basic slices in $A_1$ are negative, contributing the term $-e_1$; all basic slices in $B_2$ are negative and all basic slices in $A_1$ are positive, contributing the term $e_2$. Combining the two values with the previous range, the Euler class can take all values from $-\sum_{i=1}^{n}e_{2i-1}=2q+2$ to $\sum_{i=2}^{n}e_{2i}-e_1=-2q-2-(e_1+e_2)$, or from $-\sum_{i=2}^{n}e_{2i-1}+e_2=2q+2+(e_1+e_2)$ to $\sum_{i=1}^{n}e_{2i}=-2q-2$, with step size $2$. 

Let $k_e=\frac{1}{2}(e_1+e_2)$. To transform this expression, note that $s_1=(q/p)^a=q''/p''$ and there are $|a_m+1|$ basic slices in $B_2$. Then this parameter can be written as  
\begin{align*}
    k_e=&~|(s_1\ominus \frac{q}{p})\bigcdot \frac{0}{1}|+|(s_2\ominus \frac{q}{p})\bigcdot \frac{0}{1}|=-(\frac{q}{p}\ominus s_1)\bigcdot\frac{0}{1}+(s_2\ominus\frac{q}{p})\bigcdot\frac{0}{1}\\
    =&~(\frac{q}{p}\ominus s_1)\bigcdot\frac{0}{1}+|a_m+1|(-(s_1\bigcdot\frac{0}{1}))\\
    =&-\frac{q'}{p'}\bigcdot\frac{0}{1}+|a_m+1|(-q'')=-q'-|a_m+1|q''.
\end{align*}
Here we only discuss the case that the first continued fraction block with length $1$ appears in $P_1$ as $A_1$, which occurs when $a_m<-2$ (see the proof of Lemma~\ref{lem:length1}). The case for $B_1$ in $P_2$, which occurs when $a_m=-2$, is analogous.
\end{proof}

\begin{remark} \label{remark:rangeofke} We have $-q\leq k_e\leq -2q-2$ by the following inequalities:
\[k_e=
    \begin{cases}
    -(q'+|a_m+1|q'')=-q+(a_{m}+2)(q-q')\geq -q, &\text{if }a_m<-2,\\
    -(|b_n+1|q'+q'')=-q+(b_{n}+2)q'\geq -q, &\text{if }a_m=-2.
    \end{cases}\]
\end{remark}

We introduce another formula to compute the Euler class of contact structures on $S^1\times S^2$, via contact surgery diagrams. 
Let $L_1\sqcup \cdots\sqcup L_n\subset S^3$ be an oriented link in a contact surgery diagram, with topological surgery coefficients $p_i/q_i$ on $L_i$. Write $\tb_i$ for the Thurston–Bennequin invariant of $L_i$, $\rot_i$ for the rotation number of $L_i$, and $l_{ij}$ for the linking number between $L_i$ and $L_j$. We have a generalized linking matrix: 
    \begin{align*}
	Q:=\begin{pmatrix}
	p_1&q_2 l_{12} &\cdots&q_n l_{1n}\\
	q_1 l_{21} & p_2&&\\
	\vdots&&\ddots\\
	q_1 l_{n1}&&& p_n
	\end{pmatrix}.
    \end{align*}
Let $\mu_i$ denote a right-handed meridian of $L_i$, $i=1,\ldots,n$. Suppose that each $L_i$ is assigned with a contact $(\pm 1/n_i)$-surgery, where $n_i\in\N$. By \cite[Theorem 5.1]{dk}, we have the formula for the Poincar\'{e} dual of the Euler class:
    \begin{align} 
	\PD\big(e(\xi)\big)=\sum_{i=1}^k n_i\rot_i[\mu_i].\label{formula:PD(e)}
    \end{align}
where $[\mu_i], i=1,\ldots n$, are subject to the linear relation $Q\boldsymbol{\mu}=\mathbf{0}$, with $\boldsymbol{\mu}=([\mu_1],\ldots,[\mu_n])^T$. 

\begin{remark}\label{rmk:meridian}
    In practice, we use the contact surgery diagrams as shown in Figure~\ref{Figure:tw=0} to compute the Euler class of contact structures on $S^1\times S^2$. (See the proof of Theorem~\ref{thm:Eulerrange}.) To fix a convention, let $\mu$ be the right-handed meridian of the topmost surgery component in Figure~\ref{Figure:tw=0}. We always identify $H^2(S^1\times S^2; \Z)$ with $\Z$ by mapping the Poincar\'e dual of $[\mu]\in H_{1}(S^1\times S^2; \Z)$ to $1\in\Z$.
\end{remark}

\begin{lemma}\label{lem:halfLutz}
Suppose $L_{-P_1,-P_2}$ is a Legendrian knot which remains non-loose after arbitrarily many negative stabilizations. 
If $\tilde{\xi}_{-P_1, -P_2}$ is the contact structure obtained from $\xi_{ -P_1, -P_2}$ by the half Lutz twist on the positive transverse push-off of $L_{-P_1, -P_2}$, then \[e(\tilde{\xi}_{-P_1,-P_2})=e(\xi_{-P_1,-P_2})+2q.\]
For another Legendrian knot $L_{P_1,P_2}$ which remains non-loose after arbitrarily many positive stabilizations, if $\tilde{\xi}_{ P_1, P_2}$ is obtained from $\xi_{ P_1, P_2}$ by the half Lutz twist on the negative transverse push-off of $L_{ P_1, P_2}$, then \[e(\tilde{\xi}_{ P_1, P_2})=e(\xi_{ P_1, P_2})-2q.\]
\end{lemma}

This result can be derived from \cite[Proposition 4.3.3]{g}. Here we present an alternative proof via contact surgery diagrams.
\begin{proof}
    The Legendrian knot $L_{-P_1,-P_2}$ and contact structure $\xi_{-P_1,-P_2}$ can be represented via the contact surgery diagram as Figure~\ref{Figure:tw=0}. By the main theorem in \cite{dgs}, the contact structure $\tilde{\xi}_{-P_1,-P_2}$ can also be obtained from $\xi_{-P_1,-P_2}$ by contact $(+1)$-surgery on both $L_{-P_1,-P_2}$ and a Legendrian push-off of $L_{-P_1,-P_2}$ with two additional up-zigzags, as the right diagram in Figure~\ref{Figure:halfLutz}. 

    \begin{figure}[htb]
    \begin{overpic} 
    [scale=0.7]
    {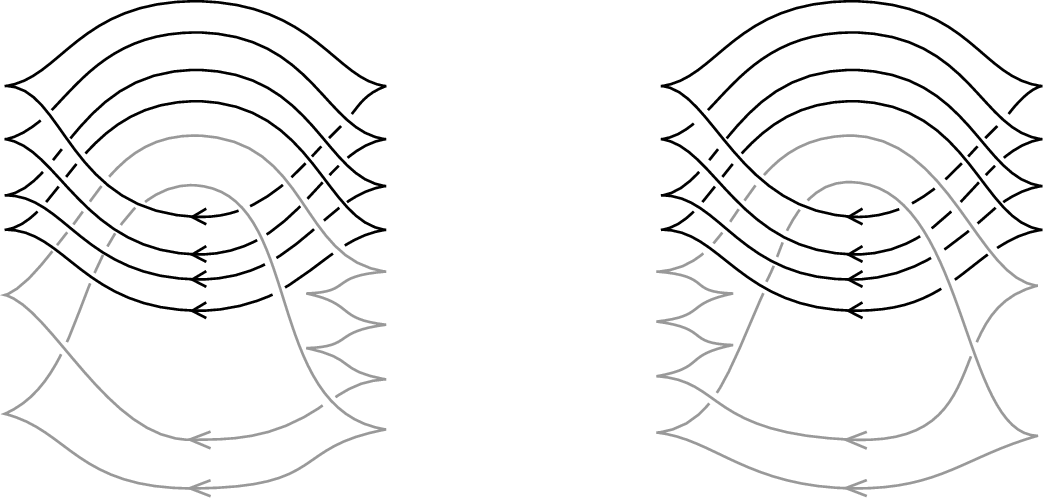}
    \put(-30,25){$(+1)$}
    \put(-30,62){$(+1)$}
    \put(-30,85){$(+1)$}
    \put(-30,100){$(+1)$}
    \put(-40,120){$(-\frac{q}{q-q'})$}
    \put(-30,138){$(-\frac{q}{q'})$}

    \put(140,20){$L_1$}
    \put(140,50){$L_2$}
    \put(140,85){$L_3$}
    \put(140,100){$L_4$}
    \put(140,116){$L_5$}
    \put(140,135){$L_6$}

    \put(188,18){$(+1)$}
    \put(188,55){$(+1)$}
    \put(188,85){$(+1)$}
    \put(188,100){$(+1)$}
    \put(178,119){$(-\frac{q}{q-q'})$}
    \put(188,138){$(-\frac{q}{q'})$}

    \put(360,20){$L_1$}
    \put(360,67){$L_2$}
    \put(360,87){$L_3$}
    \put(360,102){$L_4$}
    \put(360,118){$L_5$}
    \put(360,135){$L_6$}
    
    \end{overpic}
    \caption{Left: half Lutz twist on the positive transverse push-off of $L_{P_1,P_2}$. Right: half Lutz twist on the negative transverse push-off of $L_{-P_1,-P_2}$.}
    \label{Figure:halfLutz}
    \end{figure}
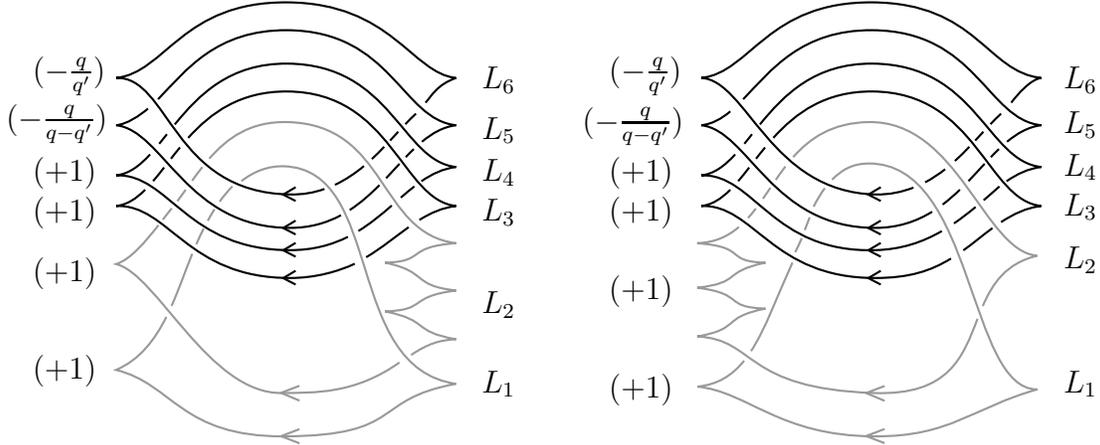
    
    Now consider the change of Euler class induced by the half Lutz twist. 
    Denote the surgery components in the left diagram of Figure~\ref{Figure:halfLutz} by $L_1,L_2,\ldots,L_6$ from bottom to top. To find the relations between $\mu_1,\mu_2$ and the meridian $\mu$ chosen in Remark~\ref{rmk:meridian}, the generalized linking matrix is 
    {\small
     \begin{align*}
	Q=\begin{pmatrix}
        0 & -1 & -1 & -1 & q-q' & q'\\
        -1 & -2 & -1 & -1 & q-q' & q'\\
	-1 & -1 & 0 & -1 & q-q' & q' \\
	-1 & -1 & -1 & 0 & q-q' & q' \\
        -1 & -1 & -1 & -1 & 2q-q' & q' \\
        -1 & -1 & -1 & -1 & q-q' & q+q'
	\end{pmatrix}.
    \end{align*}
    }\noindent
    Solving $Q\boldsymbol{\mu}=\mathbf{0}$, we obtain $[\mu_1]=-[\mu_2]=q[\mu_6]$ (here $[\mu_6]=[\mu]$). 
    
    Note that when applying formula~\ref{formula:PD(e)} to compute the Euler class, the surgery diagram in Figure~\ref{Figure:halfLutz} must first be converted into a diagram consisting of $(\pm1/n_i)$-surgery components. We then determine the relations between meridians of these components and $\mu$, and ultimately express them in terms of $\mu$.
    However, this conversion process leaves unchanged the two additional $(+1)$-surgery components $L_1$ and $L_2$, as well as the relations between their meridians and $\mu$. Therefore, to determine the change in the Euler class, it suffices to compute the contribution from these two extra components. Using the relations $[\mu_1]=q[\mu]$ and $[\mu_2]=-q[\mu]$, formula~\ref{formula:PD(e)} gives the Euler class change as $0[\mu_1]-2[\mu_2]=2q[\mu]$, which implies $e(\tilde{\xi}_{-P_1,-P_2})=e(\xi_{-P_1,-P_2})+2q$.

    According to \cite[Remark]{dgs}, performing the half Lutz twist on the negative transverse push-off of \( L_{P_1,P_2} \) yields a contact structure that can also be obtained via contact \((+1)\)-surgery on both \( L_{P_1,P_2} \) and a Legendrian push-off of \( L_{P_1,P_2} \) with two additional down-zigzags, as illustrated in the left diagram of Figure~\ref{Figure:halfLutz}. Similarly, we find that $e(\tilde{\xi}_{ P_1, P_2})=e(\xi_{ P_1, P_2})-2q$.
\end{proof}

\begin{remark}\label{rmk:sign}
    For any $2$-inconsistent pair of paths $(P_1,P_2)$, subdivide $P_1$ and $P_2$ into continued fraction blocks. By Proposition~\ref{prop:wings},  
    the $L_{P_1,P_2}$ in Lemma~\ref{lem:halfLutz} corresponds to the decoration as follows.
    \begin{itemize}
    \item When the first continued fraction block with length $1$ appears in $P_1$ as $A_1$, its sign is taken to be negative;
    \item When the first continued fraction block with length $1$ appears in $P_2$ as $B_1$, its sign is taken to be positive.
    \end{itemize}
    Hence we also know that $e(\xi_{P_1,P_2})$ is negative, by the proof of Lemma~\ref{lem:Euler}. Such a knot $L_{P_1,P_2}$ will be denoted as $L_+$ in the classification, which is represented in the left component of Figure~\ref{Figure:mountain}.
    Similarly, the $L_{-P_1,-P_2}$ is assigned the decoration with signs opposite to those described above. Then $e(\xi_{-P_1,-P_2})$ is positive and such a knot $L_{-P_1,-P_2}$ will be denoted as $L_-$ in the classification, represented in the right component of Figure~\ref{Figure:mountain}.
    
\end{remark} 

Now we give a proof of Theorem~\ref{thm:Eulerrange}, which determines all possible Euler classes of overtwisted contact structures that support non-loose torus knots. 

\begin{proof}[Proof of Theorem~\ref{thm:Eulerrange}]
We begin by reducing the classification of Euler classes for contact structures supporting non-loose Legendrian torus knots without convex Giroux torsion to the case where $\tw=0$. By Lemma~\ref{lemma:destabilization}, any such knot with $\tw<0$ can be destabilized to one with $\tw=0$. On the other hand, by Corollary~\ref{coro:tw>0stabilize}, any such knot with $\tw>0$ stabilizes to one with $\tw=0$. Therefore, both types of knots can be realized in the contact structures supporting such knots with $\tw = 0$.

According to the proof of Lemma~\ref{lemma:tw=0}, the Legendrian $(p,q)$-torus knots with $\tw=0$, $\tor=0$ and tight complement are precisely the Legendrian knots $L_{P_1,P_2}$. In Lemma~\ref{lem:numberEuler}, we determined the value set of Euler classes of contact structures $\xi_{P_1,P_2}$ supporting $L_{P_1,P_2}$. From this value set, the zero Euler class must be excluded, as it corresponds to the standard contact structure $\xi_{\text{std}}$ by Lemma~\ref{lem:Euler}.

We now consider non-loose Legendrian torus knots with convex Giroux torsion. By Lemma~\ref{lemma:addtorsion} and Lemma~\ref{lemma:tor>0}, the relevant contact structures arise from totally $2$-inconsistent path pairs by adding $l$ layers of convex Giroux torsion in the complement of $L_{P_1,P_2}$. As shown in Lemma~\ref{lem:numberEuler2}, the Euler classes corresponding to totally $2$-inconsistent path pairs are given by  
\[
e(\xi_{P_1,P_2}) = \pm 2k \quad \text{for} \quad k \in \{k_e + q + 1, k_e + q + 2, \dots, -q - 1\}.
\]
Following the convention in Remark~\ref{rmk:sign}, these contact structures may also be obtained by performing an $l$-fold Lutz twist along a negative transverse push-off of $L_{P_1,P_2}$ or a positive transverse push-off of $L_{-P_1,-P_2}$. Denote the resulting contact structures by $\tilde{\xi}_{\pm P_1,\pm P_2}$, respectively. Since a full Lutz twist preserves the Euler class, we have  
\[
e(\tilde{\xi}_{\pm P_1,\pm P_2}) = \mp 2k \quad \text{for} \quad k \in \{k_e + q + 1, k_e + q + 2, \dots, -q - 1\}.
\]
On the other hand, the effect of a half Lutz twist on the Euler class is described in Lemma~\ref{lem:halfLutz}. In this case,  
\[
e(\tilde{\xi}_{\pm P_1,\pm P_2}) = \mp(2k + 2q) \quad \text{for} \quad k \in \{k_e + q + 1, k_e + q + 2, \dots, -q - 1\}.
\]
That is, \[
e(\tilde{\xi}_{\pm P_1,\pm P_2}) = \mp2k \quad \text{for} \quad k \in \{k_e + 2q + 1, k_e + 2q + 2, \dots, - 1\}.
\]
By Remark~\ref{remark:rangeofke}, $-q\leq k_e\leq  -2q-2$. So the set $\{k_e + 2q + 1, k_e + 2q + 2, \dots, -1\}$ is a subset of  $\{q+1, q+2, \dots, -1\}$.

Finally, consider non-loose transverse torus knots. According to the proof of Corollary~\ref{coro:transversetor=0}, the contact structures supporting non-loose transverse torus knots with $\tor=0$ coincide with the contact structures $\xi_{-P_1,-P_2}$ for $2$-inconsistent pairs of paths $(P_1,P_2)$. Then, the corresponding Euler classes are always positive and the range is 
\[e(\xi_{-P_1,-P_2})= 2k \quad \text{for}\quad k\in\{1,2,\ldots,-q-1\}.\]
By Corollary~\ref{coro:transversetor>0}, the contact structures supporting non-loose transverse torus knots with $\tor>0$ coincide with $\tilde{\xi}_{-P_1-P_2}$ for totally $2$-inconsistent pairs of paths $(P_1,P_2)$. Hence the range of Euler classes is
\[
e(\tilde{\xi}_{- P_1,- P_2}) = 2k ~\text{ or }~ 2k + 2q \quad \text{for} \quad k \in \{k_e + q + 1, k_e + q + 2, \dots, -q - 1\}.
\]

In summary, if a contact structure $\xi$ on $S^1\times S^2$ contains a non-loose $(p,q)$-torus knot, then $|e(\xi)|\leq -2q-2$.
 \end{proof}

\begin{corollary}\label{coro:signoftw>0}
    The Euler class is negative for the contact structures in the left two diagrams of Figure~\ref{Figure:tw>0}, and positive for those in the right two diagrams.
\end{corollary}
\begin{proof}
In each diagram of Figure~\ref{Figure:tw>0} where $\tw(L_\pm)=i$, let $L=L_1\sqcup\cdots\sqcup L_{i+2}$ denote the oriented link with all surgery components oriented clockwise. Specifically, $L_1$, $L_{i+1}$, and $L_{i+2}$ are the contact $(+1)$-, $q'/(q'-q)$-, and $(q'-q)/q'$-surgery components, respectively. Let $\mu_i$ be the right-handed meridian of $L_i$. To compute the Euler class $e(\xi)$ of each diagram, we follow the method in  Lemma~\ref{lem:halfLutz}: First convert the diagram into a sequence of contact $(\pm1/n_k)$-surgeries, then apply Formula~\ref{formula:PD(e)}, and finally express the result in terms of $[\mu]$, which is conventionally mapped to $1\in\Z$ as per Remark~\ref{rmk:meridian}. 

Consider the diagrams in \cite[Figure 26]{emm} with the coefficient \((-\frac{p}{p'})\) replaced by \((-\frac{q}{q'})\). The final diagram in that figure corresponds to the smooth surgery description of the contact surgery diagrams shown in Figure~\ref{Figure:tw>0}. Here, \(\mu\) denotes the right-handed meridian of the surgery component in the first diagram of \cite[Figure 26]{emm}, which carries the coefficient \((-\frac{q}{q'})\) now. By tracking the Kirby moves through to the last diagram, we obtain the relation \([\mu_{i+2}] = (-1)^{i} [\mu]\).

The conversion  from the contact surgery diagrams in Figure~\ref{Figure:tw>0} to  contact $(\pm1/n_k)$-surgeries preserves the component $L_1$, whose topological surgery coefficient is $1+\tb(L_1)=-1$, as well as the relation between $\mu_1$ and $\mu$. 

Now we decompose the Euler class $e(\xi)$ into two parts: $e(\xi)=e_{L_1}+e_{L\setminus L_1}$. Here $e_{L_1}=\rot(L_1)[\mu_1]$ corresponds to the component $L_1$, and $e_{L\setminus L_1}$ is contributed by the remaining components in $L\setminus L_1$. 
To compute $e_{L_1}$, we solve the equation $Q\boldsymbol{\mu}=\mathbf{0}$, where $Q$ is the generalized linking matrix of $L$. This yields $[\mu_1]=(-1)^{i}q[\mu_{i+2}]=q[\mu]$. Consequently $e_{L_1}=\rot(L_1)q$, with $\rot(L_1)=+1$ for the left two diagrams and $\rot(L_1)=-1$ for the right two diagrams in Figure~\ref{Figure:tw>0}. Thus we have \[e(\xi)=e_{L_1}+e_{L\setminus L_1}=\rot(L_1)q+e_{L\setminus L_1}\in E=\{2q+2,2q+4,...,-2,2,...,-2q-2\},\]
where the range $E$ is derived from the proof of Theorem~\ref{thm:Eulerrange} and every value in $E$ is attainable. Observe that the range of $e_{L\setminus L_1}$ is the same for both the left and right diagrams, because the sublink $L\setminus L_1$ is identical in all four diagrams. Hence the possible values of $e_{L\setminus L_1}$ must lie in the intersection of the two shifted sets $E-q=\{e-q~|~e\in E\}$ and $E+q=\{e+q~|~e\in E\}$; that is, $$e_{L\setminus L_1}\in\{q+2,q+4,...,-q-4,-q-2\}.$$ Therefore, for the two left diagrams we have $e(\xi)=q+e_{L\setminus L_1}<0$, while for the two right diagrams $e(\xi)=-q+e_{L\setminus L_1}>0$.
\end{proof}

\section{Non-loose \texorpdfstring{$(1,q)$}{(1,q)}-torus knots} \label{Section:(1,q)}

In this section, we classify non-loose Legendrian and transverse $(1,q)$-torus knots in $S^1\times S^2$.  When $p=1$, we have $p'=1$ and $q'=q+1$.  Recall that $m(1,q)=-2q$, $n(1,q)=-q-1$, $l(1,q)=1$, $t(1,q)=-q-1$ and $k_e=-2q-2$. Then, by Theorem~\ref{thm:Eulerrange}, for all potential contact structures $\xi$ and $\xi^T$ on $S^1\times S^2$ supporting non-loose Legendrian and transverse $(1,q)$-torus knots, respectively, we have \[ e(\xi)\in\{2q+2,2q+4,\ldots,-2,2,\ldots,-2q-4,-2q-2\};\]
\[ e(\xi^T)\in\{-2,2,4,\ldots,-2q-4,-2q-2\}.\]

\subsection{Non-loose Legendrian \texorpdfstring{$(1,q)$}{(1,q)}-torus knots.}

In this subsection, we explain Example~\ref{example:(1,q)}.

First, we use the formula~\ref{formula:PD(e)} to compute Euler classes of contact structures corresponding to the contact surgery diagrams. 
Using Ding-Geiges algorithm in \cite{dg}, the contact structure in Figure~\ref{Figure:tw=0} can be transformed to those in Figure~\ref{Figure:p=1tw=0}.
For the surgery components $L_1, L_2, L_3, L_4$, 
the generalized linking matrix is 
    \begin{align*}
	Q=\begin{pmatrix}
	0 & -1 & -1 & q+1 \\
	-1 & 0 & -1 & q+1 \\
        -1 & -1 & q-1 & q+1 \\
        -1 & -1 & -1 & 2q+1
	\end{pmatrix}.
    \end{align*}
Solving $Q\boldsymbol{\mu}=\mathbf{0}$, we obtain $[\mu_1]=[\mu_2]=q[\mu_3]$ and $[\mu_3]=[\mu_4]$. Then we can compute the Euler class:
\begin{align*}
    \PD(e(\xi))=\left(r_3 \pm |q+1|\right)[\mu_4],
\end{align*}
where $r_3\in\{q+1, q+3,\ldots,-q-1\}$ is the rotation number of $L_3$ and the sign $\pm$ corresponds to the sign of stabilization of $L_4$. Notice that $[\mu_4]$ is identical to the meridian $[\mu]$ chosen in Remark~\ref{rmk:meridian}. Hence we have $e(\xi)=r_3\pm |q+1|$.

\begin{figure}[htb]
\begin{overpic} 
{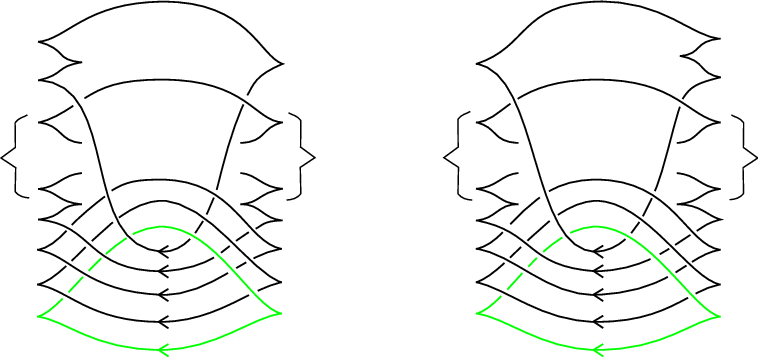}
\put(145, 10){$L_+$}
\put(-15, 33){$(+1)$} \put(145, 30){$L_1$}
\put(-15, 49){$(+1)$} \put(145, 48){$L_2$}
\put(-15, 64){$(-1)$} \put(145, 65){$L_3$}
\put(-15, 140){$(\frac{-1}{-q-1})$} \put(145, 140){$L_4$}
\put(32, 92){$\vdots$}
\put(117, 92){$\vdots$}
\put(-13, 93){$k$}
\put(158, 93){$s$}
\put(-10, 0){$e=-2k.$}

\put(360, 10){$L_-$}
\put(200, 33){$(+1)$} \put(360, 30){$L_1$}
\put(200, 49){$(+1)$} \put(360, 48){$L_2$}
\put(200, 64){$(-1)$} \put(360, 64){$L_3$}
\put(195, 140){$(\frac{-1}{-q-1})$} \put(360, 140){$L_4$}
\put(245, 92){$\vdots$}
\put(330, 92){$\vdots$}
\put(202, 93){$s$}
\put(368, 93){$k$}
\put(210, 0){$e=2k.$}
\end{overpic}



\caption{Non-loose Legendrian $(1,q)$-torus knots in some contact $S^1\times S^2$ with $\tw=0$ and $\tor=0$. The numbers of stabilization satisfy $k+s=-q-1$, where $k\in\{0,1,2,\ldots,-q-1\}$.}
\label{Figure:p=1tw=0}
\end{figure}
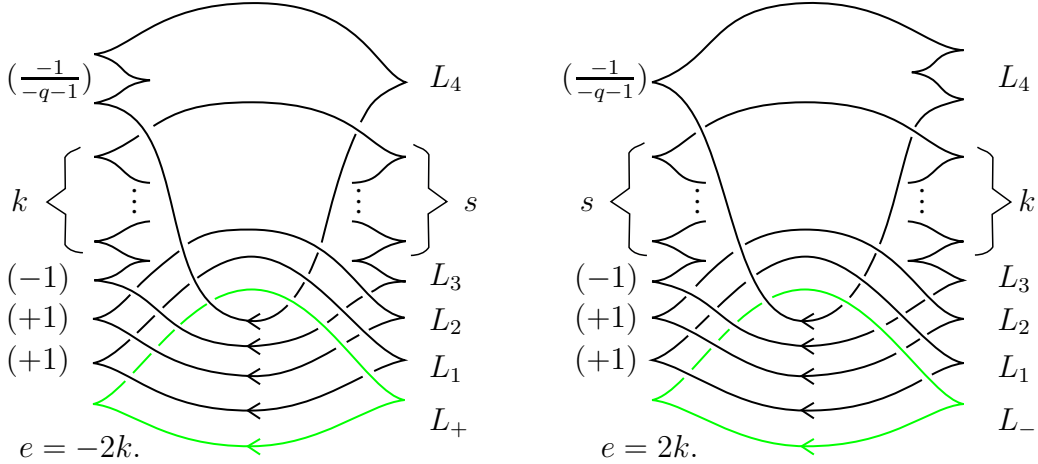

If $r_3=\pm(q+1)$, then $e(\xi)=0$. By Lemma~\ref{lem:Euler}, we know that it corresponds to the standard contact structure $\xi_{std}$ on $S^1\times S^2$. It also means that we take $k=0$ in Figure~\ref{Figure:p=1tw=0}. 

Now we classify non-loose Legendrian $(1,q)$-torus knot. The pair of paths $(P_1,P_2)$ representing $q$ and its subdivisions into continued fraction blocks are
$$P_1=A_1=\{q,\infty\}\text{ and }P_2=B_2=\{q, q+1, \ldots, -2,-1\}. $$
Then we can list all non-loose decorations of $(\overline{P_1},P_2)$ as follows:
\[\pm~(-,\overbrace{+,\ldots,+}^{k},\overbrace{-,\ldots,-}^{s}),\]
where $k\in\{1,2,\ldots,-q-1\}$ and $k+s=-q-1$. The two decorations $\pm(-,-,\ldots,-)$ for which $e=0$ are excluded, since they represent $\xi_{std}$. 

We can associate these decorations with the contact surgery diagrams in Figure~\ref{Figure:p=1tw=0} by the method in Section~\ref{section:paths}. Under the convention that $A_1$ has negative sign in $(\overline{P_1},P_2)$ as described in Remark~\ref{rmk:sign}, $(\overline{P_1},P_2)$ corresponds to $L_+$ with negative Euler class in the left diagram, while $(-\overline{P_1},-P_2)$ corresponds to $L_-$ with positive Euler class in the right diagram. From the computation above, we obtain the Euler classes of contact structures described by the two types of decorations with opposite signs listed above:
$$e=\mp2k~\text{ for }~k\in\{1,2,\ldots,-q-1\}.$$
When $q<-2$, consider those non-loose decorations with $s\geq1$, all of which are $2$-inconsistent but not totally $2$-inconsistent. Then we apply Theorem~\ref{thm:Legmountain}.

For each $e=\mp2k$ with $k\in\{1,2,\ldots,-q-2\}$, there are non-loose Legendrian torus knots $L_{\pm,k}^{i}$ for $i\in\Z$, with
$$\tw(L_{\pm,k}^{i})=i\text{ and }\tor(L_{\pm,k}^{i})=0 ,$$
such that 
$$S_{\pm}(L^{i}_{\pm,k})=L^{i-1}_{\pm,k}\text{ and }S_{\mp}(L^{i}_{\pm,k})~\text{is loose} .$$

Finally, we consider the two totally $2$-inconsistent decorations $\pm(-,+,+,\ldots,+)$. These two decorations correspond to $L^{i,0}_{\pm,-q-1}$ with $\tw=i$ and $\tor=0$ in contact structures with $e=\pm 2(q+1)$. Then we apply Lemma~\ref{lemma:addtorsion} and Lemma~\ref{lemma:tor>0}, where the corresponding Euler classes are determined in the proof of Theorem~\ref{thm:Eulerrange}.

For each $e=\pm 2(q+1)$, there are non-loose Legendrian torus knots $L_{\pm, -q-1}^{i,t}$ for $i\in \Z$ and $t\in \N\cup\{0\}$, with
$$\tw(L_{\pm, -q-1}^{i,t})=i\text{ and }\tor(L_{\pm, -q-1}^{i,t})=t,$$
such that
 $$S_{\pm}(L^{i,t}_{\pm,-q-1})=L^{i-1, t}_{\pm,-q-1}\text{ and }S_{\mp}(L^{i, t}_{\pm,-q-1})~\text{is loose}.$$ 

For each $e=\pm 2$, there are non-loose Legendrian torus knots $L_{\pm, -q-1}^{i,t+\frac{1}{2}}$ for $i\in \Z$ and $t\in\N\cup\{0\}$, with
$$\tw(L_{\pm, -q-1}^{i,t+\frac{1}{2}})=i\text{ and }\tor(L_{\pm, -q-1}^{i,t+\frac{1}{2}})=t+\frac{1}{2},$$
such that
 $$S_{\pm}(L^{i,t+\frac{1}{2}}_{\pm,-q-1})=L^{i-1, t+\frac{1}{2}}_{\pm,-q-1}\text{ and }S_{\mp}(L^{i, t+\frac{1}{2}}_{\pm,-q-1})~\text{is loose}.$$

\subsection{Non-loose transverse \texorpdfstring{$(1,q)$}{(1,q)}-torus knots} 
The Example~\ref{example:(1,q)transverse} follows immediately from Example~\ref{example:(1,q)}, Corollary~\ref{coro:transversetor=0} and Corollary~\ref{coro:transversetor>0}. In particular, $T_k$ is the positive transverse push-off of $L^0_{-,k}$ for $k\in\{1,2,\ldots,-q-2\}$ when $q<-2$, and $T^t_{-q-1}$ is the positive transverse push-off of $L_{-,-q-1}^{0,t}$ for $t\in\frac{1}{2}\N\cup\{0\}$.



\begin{thebibliography}{10}

\bibitem{c}  R. Chatterjee, {\it Non-loose Legendrian Hopf links in lens spaces}, arXiv: 2510.20779.

\bibitem{cemm}  R. Chatterjee, J. Etnyre, H. Min and A. Mukherjee, {\it Existence and construction of non-loose knots}, Int. Math. Res. Not. IMRN 2025, no. 14, rnaf210.


\bibitem{cgo} R. Chatterjee, H. Geiges and S. Onaran, {\it  Legendrian Hopf links in $L(p,1)$}. Q. J. Math. 76 (2025), no. 2, 731-756.


\bibitem{cdl} F. Chen, F. Ding and Y. Li, {\it Legendrian torus knots in $S^1\times S^2$,} J. Knot Theory Ramifications 24 (2015), no. 12, 1550064, 14 pp.

\bibitem{colin} V. Colin, {\it Chirurgies d’indice un et isotopies de sphères dans les variétés de contact tendues,} C. R. Acad. Sci. Paris Sér. I Math., 324(6):659–663, 1997.


\bibitem{dg} F. Ding and H. Geiges, {\it The diffeotopy group of $S^1 \times S^2$ via contact topology}, Compos. Math. 146, no. 4 (2010): 1096-1112.

\bibitem{dg1} F. Ding and H. Geiges, {\it A Legendrian surgery presentation of contact 3-manifolds}, Math. Proc. Cambridge Philos. Soc. 136 (2004), no. 3, 583-598.

\bibitem{dgs} F. Ding, H. Geiges and A. Stipsicz, {\it Lutz twist and contact surgery}, Asian J. Math. 9 (2005), no. 1, 57-64. 


\bibitem{dk} S. Durst and M. Kegel, {\it Computing rotation and self-linking numbers in contact surgery diagrams},  Acta Math. Hungar. 150 (2016), no. 2, 524-540.

\bibitem{e} Y. Eliashberg, { \it Classification of overtwisted contact structures on 3-manifolds}, Invent. Math., 98(3):623-637, 1989.

\bibitem{ef} Y. Eliashberg and M. Fraser. {\it  Topologically trivial Legendrian knots.} J. Symplectic Geom. (2)7 (2009), 77-127.


\bibitem{eh1} J. Etnyre and K. Honda, {\it Knots and contact geometry. I. Torus knots and the figure eight knot}, J. Symplectic Geom. 1.1 (2001): pp. 63-120. 

\bibitem{eko} J. Etnyre, M. Kegel and S. Onaran, {\it
Contact surgery numbers}, J. Symplectic Geom. 21 (2023), no. 6, 1255-1333.



\bibitem{emm} J. Etnyre, H. Min and A. Mukherjee, {\it Non-loose torus knots}, arXiv:2206.14848.

\bibitem{emtv} J. Etnyre, H. Min, B. Tosun and K. Varvarezos. {\it Tight surgeries on torus
knots}, in preparation.

\bibitem{emx} J. Etnyre, H. Min and Z. Xu, {\it Non-loose torus knots in lens spaces}, in preparation.


\bibitem{g} H. Geiges, {\it An introduction to contact topology}, Cambridge University Press, Cambridge, 2008.

\bibitem{go2} H. Geiges and S. Onaran, {\it
Legendrian rational unknots in lens spaces.}
J. Symplectic Geom. 13 (2015), no. 1, 17-50.

\bibitem{go} H. Geiges and S. Onaran, {\it Exceptional Legendrian torus knots.} Int. Math. Res. Not. IMRN (2020), no. 22, 8786–8817.

\bibitem{go1} H. Geiges and S. Onaran, {\it
Legendrian Hopf links.} Q. J. Math. 71 (2020), no. 4, 1419-1459.

\bibitem{Ha} A. Hatcher, \textit{Notes on basic 3-manifold topology}, https://pi.math.cornell.edu/~hatcher/3M/3Mfds.pdf.

\bibitem{h} K. Honda, {\it On the classification of tight contact structures. I}, Geom. Topol. 4 (2000): pp. 309-368.

\bibitem{lo} Y. Li and S. Onaran,  {\it  Strongly exceptional Legendrian connected sum of two Hopf links.} arXiv:2307.00447, to appear in Canad. J. Math.

\bibitem{lm} P. Lisca, G. Mati\'c, {\it
Tight contact structures and Seiberg-Witten invariants.}
Invent. Math. 129 (1997), no. 3, 509-525.

\bibitem{ma} I. Matkovi\u{c}, {\it
Non-loose negative torus knots.}
Quantum Topol. 13 (2022), no. 4, 669-689.

\bibitem{m} H. Min, {\it The contact mapping class group and rational unknots in lens spaces}, Int. Math. Res. Not. IMRN 2024, no. 15, 11315-11342.

\bibitem{v} T. Vogel, {\it
Non-loose unknots, overtwisted discs, and the contact mapping class group of $S^3$.}
Geom. Funct. Anal. 28 (2018), no. 1, 228-288.






\end{thebibliography}
\end{document}